\pdfoutput=1

\documentclass[final,onefignum,onetabnum]{siamart220329}
\usepackage{amsfonts}
\usepackage{amssymb}
\usepackage{esint}
\usepackage{mathtools}
\usepackage{array,makecell,longtable}
\usepackage{titling}
\usepackage{authblk}
\usepackage{commath}
\usepackage{xcolor}
\usepackage{graphicx,epstopdf}
\usepackage{bbm}
\usepackage{hyperref}
\usepackage{color,soul}
\usepackage{bm}
\usepackage{mathrsfs}

\usepackage{enumitem}
\usepackage[caption=false]{subfig}
\newsiamthm{claim}{Claim}
\newsiamremark{remark}{Remark}
\newsiamremark{example}{Example}
\usepackage{algorithmic}
\newsiamremark{exm}{Example}
\renewcommand{\thefigure}{\thesection.\arabic{figure}}

\Crefname{ALC@unique}{Line}{Lines}

\title{Determining Sources in the Bioluminescence Tomography Problem\thanks{Submitted to the editors DATE.
\funding{The work of R. Gong is supported by the National Natural Science Foundation of China (No. 12071215); The work of H. Liu is supported by the Hong Kong RGC General Research Funds (No. 11311122, 11300821 and 12301420), the NSFC/RGC Joint Research Fund (No. N\_CityU101/21), and the ANR/RGC Joint Research Grant (No. A\_CityU203/19). }}}
\author{Ming-Hui Ding\thanks{Department of Mathematics, City University of Hong Kong, Hong Kong SAR, China (\email{mingding@cityu.edu.hk}).} 
\and Rongfang Gong\thanks{School of Mathematics, Nanjing University of Aeronautics and Astronautics, Nanjing, China (\email{grf\_math@nuaa.edu.cn}).} 
\and Hongyu Liu\thanks{Department of Mathematics, City University of Hong Kong, Hong Kong SAR, China (\email{hongyu.liuip@gmail.com}, \email{hongyliu@cityu.edu.hk}).} 
\and Catharine W.K. Lo\thanks{Liu Bie Ju Centre for Mathematical Sciences, City University of Hong Kong, Hong Kong SAR, China (\email{wingkclo@cityu.edu.hk}).}}

\begin{document}

\maketitle

\begin{abstract}
    In this paper, we revisit the bioluminescence tomography\,(BLT) problem, where one seeks to reconstruct bioluminescence signals (an internal light source) from external measurements of the Cauchy data. As one kind of optical imaging, the BLT has many merits such as high signal-to-noise ratio, non-destructivity and cost-effectiveness etc., and has potential applications such as cancer diagnosis, drug discovery and development as well as gene therapies and so on. In the literature, BLT is extensively studied based on diffusion approximation (DA) equation, where the distribution of peak sources is to be reconstructed and no solution uniqueness is guaranteed without adequate a priori information. Motivated by the solution uniqueness issue, several theoretical results are explored. The major contributions in this work that are new to the literature are two-fold: first, we show the theoretical uniqueness of the BLT problem where the light sources are in the shape of $C^2$ domains or polyhedral- or corona-shaped; second, we support our results with plenty of problem-orientated numerical experiments. 
\end{abstract}

\begin{keywords}
 Bioluminescence tomography, diffusion equation, inverse source problem, solution uniqueness.
\end{keywords}

\begin{MSCcodes}
Primary 35R30; secondary 78A46, 92C55, 35Q60, 78A70
\end{MSCcodes}

\section{Introduction}

The in vivo imaging of small animals is of increasing importance in the development of modern medicine, by allowing researchers to visualise and quantify the pathophysiological and therapeutic processes occurring at the cellular and molecular levels within living organisms. However, mammalian tissues are rather opaque, and biological light sources within the tissues of small animals can only be detected externally using sensitive low-light imaging equipment. One such method is the whole-body imaging of light that is produced inside the body and transmitted through tissue, by making use of reporter genes that encode fluorescent or bioluminescent proteins. %\cite{chishima1997governing} \cite{sweeney1999visualizing}, \cite{yang2000whole}. 
These techniques have made it possible to conduct longitudinal investigations of the disease course, from early disease states through advanced disease stages.

Recently, researchers have developed bioluminescence tomography (BLT), which has proved to be a powerful tool in such in vivo biological imaging. It has already been successfully used in many areas of medicine, including to investigate tumorigenesis, cancer metastasis, cardiac diseases, cystic fibrosis, gene therapies, drug designs and many more. BLT provides distinct features that make it more advantageous in comparison to other traditional imaging modalities: it is non-invasive, highly sensitive with low signal-to-noise ratios with little background autofluorescence, and is capable of longitudinally monitoring dynamic processes %\cite{contag2002advances}, \cite{contag2002anatomy}, \cite{weissleder2003shedding}, 
\cite{ntziachristos2005looking}. This was further supported by huge advances in cooled-CCD camera technology, which has reached a level where we can detect very weak optical bioluminescence signals on the surface of a mouse body.

The first BLT prototype was conceptualised and developed by Wang and his collaborators \cite{wang2003development}, \cite{cong2005practical}, which performs quantitative 3D reconstructions of internal sources from bioluminescent views measured on the external surface of the mouse with consideration of heterogeneous scattering properties. With its strong performance and affordability, BLT has generated much interest. 

\subsection{Problem Setup and Background}

In this paper, we revisit the BLT problem, where one seeks to reconstruct the bioluminescence signals (an internal light source) from external measurements of the Cauchy data. For the Euclidean space $\mathbb{R}^n$ for $n=2,3$, let $\Omega\subset\mathbb{R}^n$ be a domain that contains the object to be imaged. Let $u(x,\vartheta,t)$ be the light flux density in the direction $\vartheta\in \mathbb{S}^{n-1}$ at $x\in\Omega$, where $\mathbb{S}^{n-1}$ is the $n$-th dimensional unit sphere. Then, the propagation of light through a random media is given by the radiative transfer equation
\begin{multline}
    \frac{1}{c} \frac{\partial u}{\partial t}(x,\vartheta,t) + \vartheta \cdot \nabla u(x,\vartheta,t) + \mu(x) u(x,\vartheta,t) \\= \mu_s(x) \int_{\mathbb{S}^{n-1}} \kappa(\vartheta \cdot \vartheta') u(x,\vartheta',t)\, d\vartheta' + q(x,\vartheta,t), \quad \\ \text{ for }t>0, x\in\Omega, \vartheta\in \mathbb{S}^{n-1}.\label{eq:RTE}
\end{multline}
Here, $c$ denotes the speed of the particle, $\mu = \mu_a + \mu_s$ with $\mu_a$ and $\mu_s$ denoting the absorption and scattering coefficients respectively, $\kappa$ being the scattering kernel normalised such that $\int_{\mathbb{S}^{n-1}} \kappa(\vartheta \cdot \vartheta') \, d\vartheta' = 1$, and $q$ is the internal light source energy density. The initial condition for $u$ is given by 
\begin{equation}\label{eq:RTEInitial}u(x,\vartheta,0) = 0 \quad \text{ for }x\in\Omega, \vartheta\in \mathbb{S}^{n-1},\end{equation}
and we prescribe the following boundary condition for $u$
\begin{equation}\label{eq:RTEBdry}\begin{cases}u(x,\vartheta,t) = g^-(x,\vartheta,t),\\
\nu(x) \cdot \vartheta \leq 0, \end{cases}\quad t>0, x\in \partial\Omega, \vartheta \in \mathbb{S}^{n-1},
\end{equation}
where $\nu(x)$ denotes the outward unit normal vector on the boundary $\partial\Omega$ of $\Omega$. 
Consequently, $g^-$ represents the incoming flux. For a more comprehensive discussion of this model, we refer readers to %\cite{Ishimaru1997}, \cite{Arridge1999}, 
\cite{NattererWubbeling2001}, %\cite{AnikonovKovtanyukProkhorov2002} 
or \cite{ArridgeSchotland2009}.

We attempt to reconstruct the internal light source $q$ from measurements of the outgoing radiation, given by 
\begin{equation}\label{eq:RTEMeasure}g(x,t) = \int_{\mathbb{S}^{n-1}} \nu(x)  \cdot \vartheta u(x,\vartheta,t) \, d\vartheta, \quad \nu(x)\cdot\vartheta>0, x\in\partial\Omega, t>0.\end{equation}
However, such a problem is very difficult, due to the complexity of the interaction terms in \eqref{eq:RTE}, although there are some efforts for solving the forward radiative transfer equations \cite{ren2019fast,ren2004algorithm}. Thus, we simplify \eqref{eq:RTE} by approximating it in the following way: 
In the biological setting, the mean-free path of the particle in biological tissues ranges between 500 to 1000 nm, which is very small in the context of radiative transfer when compared to the problem media. 
As a result, it is known (see, for instance, \cite{NattererWubbeling2001}) that the predominant phenomenon is scattering instead of transport. Therefore, we can approximate the radiative transfer equation \eqref{eq:RTE} with the diffusion equation% as in \cite{LaiLiUhlmann2019RTEDiffusion}
, which has already been widely used in optical tomography (see, for instance, %\cite{Arridge1999}, 
\cite{NattererWubbeling2001}, %\cite{AnikonovKovtanyukProkhorov2002} or 
\cite{ArridgeSchotland2009}%, or \cite{ColtonKressBook}
).  

Let $u_0$ be the diffusion approximation 
\[u_0=u_0(x,t) := \frac{1}{4\pi} \int_{\mathbb{S}^{n-1}} u(x,\vartheta,t)\, d\vartheta,\]
representing the average photon flux density in all directions, and define $q_0$ similarly by
\[q_0=q_0(x,t) := \frac{1}{4\pi} \int_{\mathbb{S}^{n-1}} q(x,\vartheta,t)\, d\vartheta.\] 
Then it can be shown that $u_0$ satisfies approximately the following initial-boundary value problem 
\begin{equation}\label{eq:Diffusion}
    \begin{cases} 
    \frac{1}{c} \frac{\partial u_0}{\partial t} -\nabla \cdot (D\nabla u_0) + \mu_a u_0 = q_0 &\quad \text{ in }\Omega \times (0,\infty),\\
    u_0 + 2D \partial_\nu u_0 = g^- &\quad \text{ on } \partial\Omega \times (0,\infty), \\
    u_0(\cdot,0) = 0 &\quad \text{ in } \Omega,
    \end{cases}
\end{equation}
where 
\[D=D(x) := \frac{1}{3(\mu_a(x) + \mu'_s(x))},\] and we used $\partial_\nu $ to denote $\frac{\partial}{\partial \nu}$ for the directional derivative $\nabla u\cdot \nu$. 
Here, in the simplification, we have omitted the refraction at the boundary, without loss of generality. Corresponding to \eqref{eq:RTEMeasure}, we measure 
\begin{equation}
    g = -D \partial_\nu u_0 \quad \text{ on } \Gamma_0 \times (0,\infty),
\end{equation}
where the measurable area $\Gamma_0$ is a part of $\partial\Omega$.

After the injection of luciferin, the bioluminescence signal varies and reaches a peak. Practically, the measurements are taken at the peak emission. Since the internal bioluminescence
distribution induced by reporter genes is relatively stable at the peak, the time dependence
is often neglected. Discarding all the time dependent terms in \eqref{eq:Diffusion}, the stationary BLT model is given by 
\begin{equation}\label{eq:StatBLT}
    \begin{cases} 
    -\nabla \cdot (D\nabla u_0) + \mu_a u_0 = q_0 &\quad \text{ in }\Omega,\\
    u_0 + 2D \partial_\nu u_0 = g^- &\quad \text{ on } \partial\Omega,
    \end{cases}
\end{equation}
with the measurement 
\begin{equation}\label{Neumann}
    g = -D \partial_\nu u_0 \quad \text{ on } \Gamma_0.
\end{equation}
Define the forward operator $F$ mapping $q_0$ to $g$: 
\begin{equation}\label{eq:mea}
    g=F(q_0):=-D \partial_\nu u_0.
\end{equation}
$\Lambda$ is the inversion of $F$:
\begin{equation}
    \Lambda : g\to \, q_0.
\end{equation}

In the following, for the statement of simplicity and without loss of the generality, let $\Gamma_0=\partial \Omega$. Moreover, we set $g^-=0$ which signifies the imaging process taking place within a dark environment. Then by combining the measurement and the Robin boundary condition in \eqref{eq:StatBLT} , a Dirichlet boundary condition for $u_0$ is obtained
\begin{equation}
u_0 =  2g \text{ on }\partial\Omega.
\end{equation}

As a result, the BLT problem can be stated as follows: Given a measurement of the outgoing flux $g$ on $\partial\Omega$, find a source $q_0$ of peak time as well as the corresponding photon flux $u_0$ satisfying 
\begin{equation}\label{eq:mainPb}
    \begin{cases} 
    -\nabla \cdot (D\nabla u_0) + \mu_a u_0 = q_0 &\quad \text{ in }\Omega,\\
    u_0 + 2D \partial_\nu u_0 = 0 &\quad \text{ on } \partial\Omega, \\
    u_0 = 2g &\quad \text{ on } \partial\Omega.
    \end{cases}
\end{equation}

\subsection{Discussion and Organisation of this Paper}

The BLT problem formulated as the inverse source problem in \eqref{eq:mainPb} has attracted much attention, due to its many applications in biological imaging. It is known that such a problem is non-unique in general \cite{BLT2003} %\cite{han2006mathematical}
, though one can obtain uniqueness with adequate a priori knowledge. However, to date, the only known theoretical uniqueness result was obtained by Wang, Li and Jiang in \cite{BLT2003}. In that work, the authors derived uniqueness results only in the case where the sources are spatially separated and do not show any form of congregation, i.e. they can be represented by Kronecker $\delta$ functions, or in the case where the sources gather in the form of solid or hollow balls (up to an integral equation), with their intensities known. 

Numerically, more often, a permissible source region $\Omega_0$ for $q_0$ is given to weaken the nonuniqueness. For instance, $\Omega_0$ can be estimated by other molecular imaging methods such as MRI \cite{zhang2014incorporating}. Another frequently adopted strategy is to assume the source sparsity by noticing that the support of the source is relatively small when compared with the problem domain $\Omega$, see \cite{he2010sparse}%, \cite{guo2012efficient}
 for instance.

In addition, it is indicated that introducing light spectral information in BLT could weaken the ill-posedness theoretically and improve the solution accuracy numerically in the sense that on one hand, spectral-dependent optical parameters make the models more accurate; on the other hand, the filtered multispectral data provide more than one boundary measurement \cite{chaudhari2005hyperspectral} \cite{dehghani2008spectrally}.
Also, in \cite{bal2016ultrasound}, \cite{gong2016analysis}, the angular dependent data is measured. Theoretically, in the case that the true source is assumed to be independent of angular variable $\vartheta$, using angular-dependent measurements on the boundary could lead to the solution uniqueness \cite{bal2016ultrasound}. However, practically, on one hand, the associated stationary RTE is difficult to solve numerically, especially when the anisotropy factor takes values near $\pm 1$; on the other hand, it is hard to obtain the angular-dependent data because only integrated information is measured. Therefore, it is still challenging to study the (stationary) RTE-based BLT problem.

In this work, for the DA-based BLT problem \eqref{eq:mainPb}, we generalise the results of \cite{BLT2003}, by considering two different types of mass sources, namely when the mass is in the form of multiple disconnected $C^2$ domains, which may be embedded in each other, or when the mass is in the form of multiple disconnected domains that are in the form of polyhedrons or coronas which are not smooth and possess corner. Furthermore, we can also determine both the support and the physical intensity of the light sources, partially in the general case and fully if the intensities are of a particular (still general) form. This is done by using a single boundary measurement.

Having shown the theoretical uniqueness of the BLT problem, we will support our result with some numerical experiments. Similar results have been obtained for the two cases in \cite{BLT2003} which were mentioned above, including the recovery of two-dimensional point sources in %\cite{li2006two}, 
\cite{han2007bioluminescence} and %\cite{han2009integrated}, 
\cite{gong2014fast} and three-dimensional point sources in \cite{cong2006boundary}% and \cite{lv2006multilevel}, \cite{wang2006first},  \cite{jiang2007image}, \cite{cong2010differential}
, and the recovery of spherical light sources in %\cite{cong2006born}, %\cite{han2006mathematical}, \cite{lv2006multilevel},  \cite{han2009integrated} 
%and  
\cite{gong2014fast}. Furthermore, numerical experiments have also been extended and conducted for other types of domains, proving the wider use of BLT. This includes two-dimensional polygonal domains in \cite{cong2005practical}, three-dimensional small polyhedral domains in %\cite{cong2006born}, %\cite{cong2006multispectral} 
 \cite{han2006mathematical2}, and arbitrary light source functions in \cite{cheng2009numerical}, which further supports our theoretical result on polyhedral domains we present in this paper. These previously known results combined with our numerical experiments verifying the uniqueness of the reconstruction of light sources in the shape of smooth domains or with corners showcase the potential of BLT in real biological settings.

In summary, we list the major contributions of this work in what follows:
\begin{enumerate}[label=(\roman*)]
    \item We establish uniqueness results in recovering several general light source domains in several separate cases by a single boundary measurement. These results are highly interesting, in particular in the following two aspects. First, to our best knowledge, this is the first result in the literature concerning the shape determination of general domains by a single measurement. The existing study only shows the result for point sources or spherical sources. Second, this is achieved via a single measurement, which is usually the case in biological imaging.
    \item In achieving the results in (i), we need to impose strong a priori information on the target light source domain. In the case of more regular domains, we require the light source and the domain of analysis to have surfaces which are at least $C^2$-smooth. In the case of domains with corners, we assume that the shape belongs to certain admissible classes, which are general enough to include some physically important cases including the polygonal/polyhedral case as verified previously in \cite{cong2005practical}, %\cite{cong2006born} %, \cite{cong2006multispectral}
    and \cite{han2006mathematical2}. At the same time, this assumption further verifies that a priori information can bring beneficial advantages to the inversion process in the theory of inverse problems.
    \item Moreover, our results are verified using numerical experiments. Our numerical experiments support our theoretical results in both the case of smooth domains and the case of polyhedral domains, in both two dimensions and three dimensions. These results generalise previous numerical results which only considered point sources (see for instance %\cite{li2006two}, 
    \cite{han2007bioluminescence}, %\cite{cong2010differential}, 
    \cite{gong2014fast}). Furthermore, they reinforce previous known numerical results for domains similar to the ones we consider, including spherical domains (see for instance %\cite{cong2006born}, %\cite{lv2006multilevel},  \cite{han2009integrated}, 
    \cite{gong2014fast}) for the former, and polyhedral domains (see for instance \cite{cong2005practical}, %\cite{cong2006multispectral}, 
    \cite{han2006mathematical2}) for the latter.
\end{enumerate}

The rest of the paper is organised as follows. In Section \ref{sect:Prelims}, we provide rigorous mathematical formulations for the setup of the BLT problem, and also recall some known uniqueness / non-uniqueness results. We will extend these results to general $C^2$-smooth domains in Section \ref{sect:smooth}, and to non-smooth polygonal/polyhedral or corona-shape domains by conducting a microlocal characterisation of corner singularities in Section \ref{sect:polygon}. These results are then verified by numerical experiments, of which the numerical algorithm is given in Section \ref{sec:algo}. The corresponding numerical results and discussion are provided in Section \ref{sec:resu}, for both $C^2$-smooth domains and non-smooth polygonal/polyhedral or corona-shape domains.

\section{Preliminaries}\label{sect:Prelims} 

For the problem \eqref{eq:mainPb}, in the rest of the article, we drop all the subscripts in \eqref{eq:mainPb} for simplicity. Furthermore, we assume that $0<D_*<D<D^*$, $\mu\geq0$ are bounded functions. Assume further that $D\in C^1(\Omega)$ is sufficiently regular near $\partial\Omega$. 

We first recall the standard regularity results for second order elliptic equations.
\begin{theorem}[see for instance %\cite{EvansBook} or 
\cite{LadyzhenskayaUraltsevaBook}]
    For any $q\in L^2(\Omega)$, the weak solution $u\in H^1(\Omega)$ to \eqref{eq:StatBLT} is such that 
    \[u\in H^2_{loc}(\Omega),\]
    and for every open subset $U\Subset\Omega$, the estimate 
    \[\norm{u}_{H^2(U)} \leq C(\norm{q}_{L^2(\Omega)} + \norm{u}_{L^2(\Omega)})\]
    holds for some constant $C$ depending only on $U$, $\Omega$, $D$ and $\mu$.

    Suppose further that $\Omega$ is such that $\partial\Omega$ is $C^2$. Then the weak solution $u\in H^1(\Omega)$ is such that 
    \[u\in H^2(\Omega)\] and satisfies the estimate \[\norm{u}_{H^2(\Omega)} \leq C(\norm{q}_{L^2(\Omega)} + \norm{u}_{L^2(\Omega)} + \norm{g}_{H^{-1/2}(\partial\Omega)} + \norm{g}_{H^{1/2}(\partial\Omega)})\] for a different constant $C$ depending only on $\Omega$, $D$ and $\mu$.
\end{theorem}

Here $L^2(\Omega)$ denotes the Hilbert space 
\[L^2(\Omega) := \left\{f:\int_\Omega|f(x)|^2\,dx<\infty\right\}\] 
with inner product 
\[\langle f,g\rangle = \int_\Omega f(x)g(x) \,dx;\] 
$H^1(\Omega)$ denotes the Hilbert space 
\[H^1(\Omega) := \left\{f\in L^2(\Omega):\nabla f\in L^2(\Omega)\right\}\] 
with inner product 
\[\langle f,g\rangle = \int_\Omega fg + \nabla f \cdot \nabla g \,dx;\] 
and $H^2(\Omega)$ denotes the Hilbert space 
\[H^2(\Omega) := \left\{f\in H^1(\Omega):\nabla f\in L^2(\Omega)\right\}\] 
with inner product 
\[\langle f,g\rangle = \int_\Omega fg + \nabla f \cdot \nabla g + \Delta f \Delta g \,dx.\] 
The subspaces $H^1_0(\Omega)$ and $H^2_0(\Omega)$ of $H^1(\Omega)$ and $H^2(\Omega)$ are given by the closure of smooth functions with compact support inside $\Omega$, i.e. $C_c^\infty(\Omega)$ functions, in the spaces $H^1(\Omega)$ and $H^2(\Omega)$ with respect to their corresponding norms, respectively. Furthermore, there is a one-to-one trace map from $H^k(\Omega)$ to the trace space $H^{k-\frac{1}{2}}(\partial\Omega)$, and the subspace $H^k_0(\Omega)$ is the restriction of $H^k(\Omega)$ to functions with vanishing trace (see, for instance, \cite{LionsMagenesBook1}% or \cite{Demengel}
) for an integer $k$.

Then, it is known that the BLT problem \eqref{eq:mainPb} does not have a unique solution in general.

\begin{theorem}[Theorem IV.1 of \cite{BLT2003}]
    Suppose that the BLT problem \eqref{eq:mainPb} has a solution. Then there exists infinitely many solutions, given by 
    \[q = q_H - \nabla \cdot (D \nabla m) + \mu m \quad \text{ for any }m \in H^2_0(\Omega),\] where $q_H$ is the representative solution with minimal $L^2$ norm.
\end{theorem}

Given this nonuniqueness result in the general case, in this work, we restrict our consideration to bioluminescent source distributions in a certain parametrised form, so that the solution uniqueness may be established in that specific case. In this case, the BLT problem is given by the following:
\begin{equation}\label{eq:main}
    \begin{cases} 
    -\nabla \cdot (D\nabla u) + \mu u = q=\varphi \chi_\omega &\quad \text{ in }\Omega,\\
    u + 2D \partial_\nu u = 0  &\quad \text{ on } \partial\Omega, 
    \end{cases}
\end{equation}
and
\begin{equation}\label{eq:main1}
    u = 2g \quad \text{ on } \partial\Omega,
\end{equation}
for $\omega\Subset\Omega$. In this case, $\Lambda$ is reduced to % with the measurement map
\begin{equation}\label{eq:map}
    \Lambda := g \to (\omega,\varphi),
\end{equation}
and the following result is known:
\begin{theorem}[Theorems IV.2 and IV.3 of \cite{BLT2003}]\label{KnownThm} Suppose that $D$ and $\mu$ are piecewise constant. 

    (i) Suppose that \[\omega=\bigcup_{j=1}^m \delta(x-x_j) \quad \text{ and } \quad \varphi(x) = \begin{cases}
        \varphi_j \text{ constant }& \text{ if }x=x_j,\\
        0 & \text{ otherwise.}
    \end{cases}\]
    Then $\omega$ and $\varphi$ are uniquely determined by a single boundary measurement $g$.

    (ii) Suppose that 
    \[\omega=\bigcup_{j=1}^J B_{r_j,R_j}(x_j) \quad \text{ and } \quad \varphi(x) = \begin{cases}
        \varphi_j \text{ constant } & \text{ if }x\in B_{r_j,R_j},\\
        0 & \text{ otherwise,}
    \end{cases}\]
    where $B_{r_j,R_j}(x_j)$ are hollow spheres centred at $x_j$ with radius $R_j$ and hole of radius $r_j$. 
    Then $\omega$ and $\varphi$ are uniquely determined (up to an integral equation) by a single boundary measurement $g$.
\end{theorem}
In the next Sections, we will extend these results to two new cases, the case where $\omega$ is a $C^2$ smooth domain which generalise the results of Theorem \ref{KnownThm}, and the case where $\omega$ is polyhedral-shaped.

\section{Smooth Domains}\label{sect:smooth}

We show a unique recovery result for the inverse problem \eqref{eq:main}--\eqref{eq:main1} in the case of $C^2$ domains. Before that, we introduce an admissibility condition for $q$.

\begin{definition}
    Let $\Omega$ be an open bounded set in $\mathbb{R}^n$, $n=2,3$, such that $\partial\Omega\in C^2$. We say that $q\in L^2(\Omega)$ is admissible and write $q\in \mathcal{A}$ if $q$ is of the form $q=\varphi\chi_{\omega}$ for the open bounded subset $\omega\Subset\Omega$ such that $\partial\omega\in C^2$, $\Omega\backslash \omega$ is connected, and $q\not\equiv0$ on $\partial\omega$.
\end{definition}

Then, our first main result is the following

\begin{theorem}\label{thm:mainSmooth}
    Let $\Omega$ be an open bounded set in $\mathbb{R}^n$, $n=2,3$, with $C^2$ boundary. Suppose $q\in\mathcal{A}$ is a solution to the BLT problem \eqref{eq:main}--\eqref{eq:main1}. Then $q$ is uniquely determined by a single boundary measurement $g$, in the sense that the smooth domain $\omega$ and light intensity $\varphi(x)$ are uniquely determined on the surface $x\in\partial\omega$.
\end{theorem}

We first begin with an auxiliary result. 

\begin{theorem}\label{thm:auxSmooth}
For $n=2,3$, let $\omega\subset\Omega\subset\mathbb{R}^n$ be open bounded domains with $C^2$ boundaries, such that the complement $\Omega\backslash\omega$ is connected. For the solution $q=\varphi \chi_\omega$ to the BLT problem \eqref{eq:main}--\eqref{eq:main1}, let $\varphi\in L^2(\Omega)$ such that $\varphi(x)\neq 0$ for some $x\in \partial\omega_c$ for some component $\omega_c$ of $\omega$. Then, the solution $u\in H^2_{loc}(\Omega)$ to \eqref{eq:main} is such that $u|_{\partial\Omega}$, $\partial_\nu u|_{\partial\Omega}$ cannot be identically zero.
\end{theorem}

\begin{proof}
First consider $\omega$ with a single component. Suppose on the contrary that $u|_{\partial\Omega}=\partial_\nu u|_{\partial\Omega}=0$. By the unique continuation principle for elliptic equations \cite{KochTataru2001UCPElliptic} and the connectedness of $\Omega\backslash\bar{\omega}$, we have that $u=0$ in $\Omega\backslash\bar{\omega}$. Since $\partial\omega\in C^2$, $u|_\omega\in H^2_0(\omega)$. In particular, $u|_{\partial\omega}=0$. 
However, since $u$ satisfies \eqref{eq:main} with $\varphi(x)\neq 0$ for some $x\in \partial\omega$, we arrive at a contradiction.

Next, we consider the case when $\omega$ consists of multiple components. Let $\omega_c\subset\omega$ be any component such that $\overline{\omega_c} \cap \overline{\omega \backslash \omega_c} = \emptyset$. Denote by $u_c\in H^2_{loc}(\Omega)$ the solution satisfying  \[-\nabla \cdot (D\nabla u_c) + \mu u_c = \varphi \chi_{\omega_c}.\]
Note that $u_c$ may not be unique.
Set $u=\sum u_c$. We will prove that $u|_{\partial\Omega}=\partial_\nu u|_{\partial\Omega}=0$ if and only if the individual boundary value functions vanish, i.e. $u_c|_{\partial\Omega}=\partial_\nu u_c|_{\partial\Omega}=0$ for each $c$. Indeed, consider a component $\omega_c$ of $\omega$, and let $u_c\in H^2_{loc}(\Omega)$ satisfy 
\[-\nabla \cdot (D\nabla u_c) + \mu u_c = \varphi \chi_{\omega_c},\quad u_c|_{\partial\Omega} = \partial_\nu u_c|_{\partial\Omega} = 0.\] By linearity, taking the sum of all the equations of $u_c$, we have that
\[-\nabla \cdot (D\nabla u) + \mu u = \varphi \chi_{\omega}, \quad u|_{\partial\Omega} = \partial_\nu u|_{\partial\Omega} = 0.\] 
Conversely, if $u$ is such that 
\[-\nabla \cdot (D\nabla u) + \mu u = \varphi \chi_{\omega}, \quad u|_{\partial\Omega} = \partial_\nu u|_{\partial\Omega} = 0,\] 
then by the unique continuation principle (since $\omega$ and $\Omega$ have sufficiently smooth boundaries), $u=0$ in $\Omega\backslash\bar{\omega}$, so in particular, the restriction of $u$ to the subset $\omega_c$ is such that $u|_{\omega_c}\in H^2_0(\omega_c)$.

Let 
\[\tilde{u} = \begin{cases}
    u &\quad \text{ in }\omega_c,\\
    0 &\quad \text{ in }\Omega\backslash\omega_c.
\end{cases}\]
Then, since $u|_{\omega_c}\in H^2_0(\omega_c)$ and $-\nabla \cdot (D\nabla u) + \mu u = \varphi \chi_{\omega_c}$ in a neighbourhood of $\overline{\omega_c}$ (which does not intersect $\omega\backslash\omega_c$, we see that 
\[-\nabla \cdot (D\nabla \tilde{u}) + \mu \tilde{u} = \varphi \chi_{\omega_c} \quad\text{ in }\Omega, \quad \quad \tilde{u}|_{\partial\Omega} = \partial_\nu \tilde{u}|_{\partial\Omega} = 0.\] 
Therefore $\tilde{u}$ and $u_c$ solve the same elliptic Dirichlet problem, whose solution is known to be unique, so $u_c=\tilde{u}$ and the boundary value function vanishes for the component $c$.

Therefore, we can consider each component $\omega_c$ individually and obtain the desired result.

\end{proof}

\begin{proof}[Proof of Theorem \ref{thm:mainSmooth}]
    Suppose on the contrary that $q\in\mathcal{A}$ is not unique, i.e. there exists two solutions $q=\varphi\chi_{\omega}$ and $\hat{q}=\hat{\varphi}\chi_{\hat{\omega}}$ to the BLT problem such that  $\varphi,\hat{\varphi}\in L^2(\Omega)$, and either $\varphi(x)\neq0$ for some $x\in\partial\omega \backslash \partial\hat{\omega}$ or $\hat{\varphi}(x')\neq0$ for some $x'\in\partial\hat{\omega} \backslash \partial\omega$. Without loss of generality, we take the second case. 
    As in the previous proof, we first consider the case when $\omega$ and $\hat{\omega}$ have a single component. Since $q$ and $\hat{q}$ solve the BLT problem, their respective solutions $u$ and $\hat{u}$ satisfy 
    \[u=\hat{u},\quad \partial_\nu u = \partial_\nu \hat{u} \quad \text{ on }\partial\Omega.\] 
    In particular, $\tilde{u}:=u-\hat{u}$ is such that $\tilde{u}|_{\partial\Omega} = \partial_\nu\tilde{u}|_{\partial\Omega} =0$. Furthermore, $\tilde{u}$ solves 
    \[-\nabla \cdot (D\nabla \tilde{u}) + \mu \tilde{u} = 0 \quad \text{ in }\Omega\backslash(\overline{\omega} \cup \overline{\hat{\omega}}).\] Observe that $\Omega\backslash(\omega\cup\hat{\omega})$ is connected, since it is the intersection of the connected subsets $\Omega\backslash \omega$ and $\Omega\backslash\hat{\omega}$ for the open bounded subsets $\omega,\hat{\omega}\Subset\Omega$. 
    By the unique continuation principle for elliptic equations, we have that $\tilde{u}=0$ in $\Omega\backslash(\omega\cup\hat{\omega})$, and so $\tilde{u}\in H^2_0(\Omega\backslash(\omega\cup\hat{\omega}))$ by the $C^2$ regularity of the boundary $\partial \omega\cup \partial\hat{\omega}$. 

    Now consider $\partial(\Omega\backslash(\omega\cup\hat{\omega}))$. Since $\tilde{u}=0$ there, $\tilde{u}$ satisfies 
    \begin{equation}\label{eq:SmoothPf1}
    0 = -\nabla \cdot (D\nabla \tilde{u}) + \mu \tilde{u} = \varphi\chi_{\omega} - \hat{\varphi}\chi_{\hat{\omega}}\quad \text{ on } \partial(\Omega\backslash(\omega\cup\hat{\omega})).
    \end{equation}
    In particular, for some  $x'\in \partial\hat{\omega}\backslash\partial\omega \subset \partial(\Omega\backslash(\omega\cup\hat{\omega}))$, $\varphi(x')=0$ since $x'\not\in\omega$ but $\hat{\varphi}(x')\neq0$ by assumption. This contradicts \eqref{eq:SmoothPf1}.
    Because by assumption, there exists an $x\in \partial(\Omega\backslash(\omega\cup\hat{\omega}))$ such that either $\varphi(x)\neq0$ or $\hat{\varphi}(x)\neq0$. Therefore, it must be that $\partial\omega = \partial\hat{\omega}$, i.e.  
    \[\omega = \hat{\omega}.\]
    Then, \eqref{eq:SmoothPf1} can be rewritten as 
    \begin{equation*}
    0 = -\nabla \cdot (D\nabla \tilde{u}) + \mu \tilde{u} = (\varphi - \hat{\varphi})\chi_{\omega} \quad \text{ on } \partial(\Omega\backslash \omega),
    \end{equation*}
    from which we can conclude \[\varphi = \hat{\varphi} \text{ on }\partial\omega.\]

    The case where $\omega$ has multiple components follows similarly as in the proof of the previous result.
\end{proof}

In the case where the light intensity $\varphi$ is constant, we have the following corollary.

\begin{corollary}
    Let $\Omega$ be an open bounded set in $\mathbb{R}^n$, $n=2,3$, with $C^2$ boundary. Suppose $q\in\mathcal{A}$ is a solution to the BLT problem \eqref{eq:main}--\eqref{eq:main1}, such that $q=\varphi\chi_{\omega}$ is such that $\varphi$ is constant. Then $q$ is uniquely determined by a single boundary measurement $g$, in the sense that the smooth domain $\omega$ and the constant light intensity $\varphi$ are uniquely determined.
\end{corollary}

\begin{proof}
    This follows simply from the Theorem \ref{thm:mainSmooth}, because we can uniquely determine $\varphi$ from its value on $\partial\omega$ in the case where $\varphi$ is constant.
\end{proof}

Moreover, we have the result for multiple embedded domains, in the case where $q$ is piecewise constant. Such domains are defined as follows:
\begin{definition}\label{def:nest}
    We say that $\omega$ has a \emph{nest partition} if there exists 
    \[\omega_N\Subset \omega_{N-1} \Subset \cdots \Subset \omega_2 \Subset \omega_1 = \omega \Subset \Omega \subset \mathbb{R}^n,\] where $\omega_\ell$, $\ell = 1, 2, \dots, N$, $N\in\mathbb{N}$, is such that each $\omega_\ell$ is an open bounded $C^2$ domain with $\Omega\backslash\omega_{\ell}$ connected. An example of such a nest partition is in Figure \ref{fig:NestPart}.
\end{definition}
\begin{figure}
    \centering
    \includegraphics[width=0.28\textwidth]{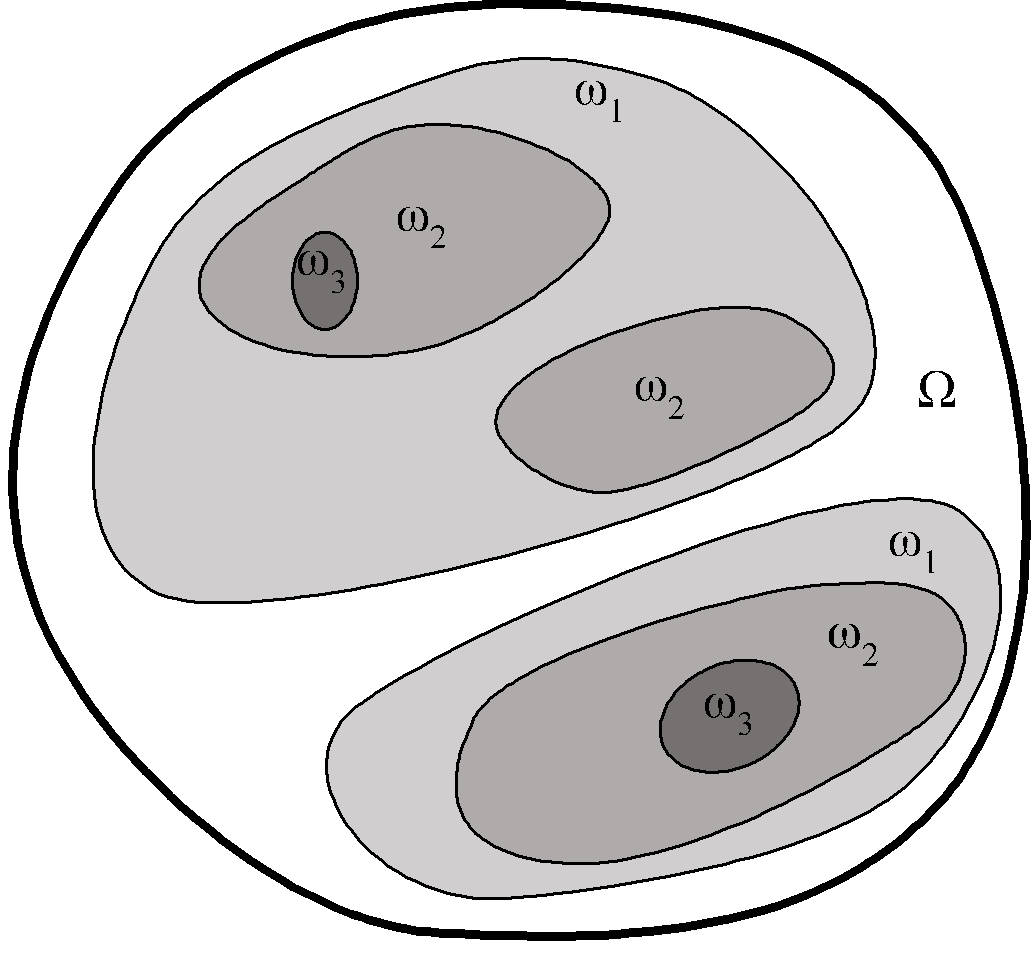}
    \caption{An example of a nest partition}
    \label{fig:NestPart}
\end{figure}

We assume that $q$ in \eqref{eq:main}--\eqref{eq:main1} fulfills the admissibility condition below. 
\begin{definition}
    Let $\Omega$ be an open bounded set in $\mathbb{R}^n$, $n=2,3$, such that $\partial\Omega\in C^2$. We say that $q$ is admissible and write $q\in \mathcal{B}$ if $q$ is of the form $q=\sum_{j=1}^m \varphi_j\chi_{\omega_j}$ for the nest partition (see Figure \ref{fig:NestPart})  
    \[\omega_m\Subset \omega_{m-1} \Subset \cdots \Subset \omega_2 \Subset \omega_1 = \omega, \] such that $\varphi_j$ is constant for each $j$. Without loss of generality, we assume that $\varphi_1\neq0$ and $\varphi_j\neq \varphi_{j+1}$ for each $j$.
\end{definition}

\begin{remark}
    We remark that the assumptions on $q$ simply mean that it is piecewise constant. 
    
    This includes the case in Figure \ref{fig:NestPart}, where $\omega_1$ has two components, $\omega_2$ has three components, and $\omega_3$ has two components, as long as the compactness assumption in Definition \ref{def:nest} is satisfied.
    
    This also includes the case where $\omega$ is a hollow ball as given in Theorem \ref{KnownThm}(ii), by taking $m=2$, $\omega_1=\bigcup_{j=1,\dots,J} B_{R_j}(x_j) \triangleq  B_{R}(x)$ and $\omega_2=B_r(x)$ and $\varphi_2 = 0$ for the balls $B_{s}(y)$ for radius $s$ and centre $y$, with $R=(R_1,\dots,R_J)$, $r=(r_1,\dots,r_J)$, $x=(x_1,\dots,x_J)$.
\end{remark}

\begin{corollary}\label{cor:SmoothEmbed}
    Let $\Omega$ be an open bounded set in $\mathbb{R}^n$, $n=2,3$, with $C^2$ boundary. Suppose  $q\in\mathcal{B}$ is a solution to the BLT problem \eqref{eq:main}--\eqref{eq:main1}. Then $q$ is uniquely determined by a single boundary measurement $g$, in the sense that the smooth domains $\omega_j$ and the constant light intensities $\varphi_j$ are uniquely determined for all $j=1,\dots,m$.
\end{corollary}

\begin{proof}
    Once again, we will show by contradiction. Assume on the contrary that $q,\hat{q}\in\mathcal{B}$ satisfy the BLT problem with the given assumptions, with the form $q=\sum_{j=1}^m \varphi_j\chi_{\omega_j}$ and $\hat{q}=\sum_{j=1}^M\hat{\varphi}_j\chi_{\hat{\omega}_j}$. Then, since $q,\hat{q}\in L^2(\Omega)$ with $q|_{\partial\omega_1}=\varphi_1,\hat{q}|_{\partial\hat{\omega}_1}=\hat{\varphi}_1$ such that $\varphi_1,\hat{\varphi}_1\neq0$ by assumption, using Theorem \ref{thm:mainSmooth}, we have that $\omega_1 = \hat{\omega}_1$ and $\varphi_1 = \hat{\varphi}_1\neq0$ on $\partial \omega_1$. Since $\varphi_1,\hat{\varphi}_1$ are constants by the assumption $q\in\mathcal{B}$, we have that $\varphi_1 = \hat{\varphi}_1\neq0$ on $\omega_1\backslash\overline{\omega}_2$.

    We will show that $\omega_j=\hat{\omega}_j$ and $\varphi_j=\hat{\varphi}_j$ for every $j=1,\dots,m$ by induction. Suppose this holds for $j=1,\dots,\ell$. Then, for the solutions $u$ and $\hat{u}$ corresponding to $q,\hat{q}$ respectively, their difference $\tilde{u}:=u-\hat{u}$ is such that $\tilde{u}|_{\partial\Omega} = \partial_\nu\tilde{u}|_{\partial\Omega} =0$, and solves 
    \[-\nabla \cdot (D\nabla \tilde{u}) + \mu \tilde{u} = 0 \quad \text{ in }\Omega\backslash(\overline{\omega_{\ell+1}} \cup \overline{\hat{\omega}_{\ell+1}}).\] 
    By the unique continuation principle and the connectedness of $\Omega\backslash(\omega_{\ell+1} \cup \hat{\omega}_{\ell+1})$, we have that $\tilde{u}=0$ in $\Omega\backslash(\omega_{\ell+1} \cup \hat{\omega}_{\ell+1})$, and so $\tilde{u}\in H^2_0(\Omega\backslash(\omega_{\ell+1} \cup \hat{\omega}_{\ell+1}))$ by the $C^2$-regularity of the boundary. 

    Then, as before, consider $\partial(\Omega\backslash(\omega_{\ell+1} \cup \hat{\omega}_{\ell+1}))$. Since $\tilde{u}=0$ there, $\tilde{u}$ satisfies 
    \[0 = -\nabla \cdot (D\nabla \tilde{u}) + \mu \tilde{u} = \varphi_{\ell+1}\chi_{\omega_{\ell+1}} - \hat{\varphi}_{\ell+1}\chi_{\hat{\omega}_{\ell+1}}\quad \text{ on } \partial(\Omega\backslash(\omega_{\ell+1} \cup \hat{\omega}_{\ell+1})).\] 
    This implies that
    \[\omega_{\ell+1} = \hat{\omega}_{\ell+1} \quad \text{ and } \quad \varphi_{\ell+1} = \hat{\varphi}_{\ell+1} \quad \text{ on }\partial\omega_{\ell+1}.\] Since $\varphi_{\ell+1},\hat{\varphi}_{\ell+1}$ are constants, this holds in $\omega_{\ell+1}\backslash\overline{\omega}_{\ell+2}$. Repeating inductively, we have the result for $j=1,\dots$ and $m=M$.
    
    Finally, the case where $\omega$ and $\hat{\omega}$ have multiple components follows similarly as in the proof of Theorem \ref{thm:auxSmooth}.
\end{proof}

\section{Polyhedral Domains}\label{sect:polygon}

Next, we consider the case when the domain is less smooth, such that $\omega$ is a polyhedral-shaped bounded Lipschitz domain, such that $\Omega\backslash\bar{\omega}$ is connected. We assume that $D$ and $\mu$ are such that 
\begin{equation}\label{def:k}k:= - \frac{\Delta \sqrt{D}}{\sqrt{D}} - \frac{\mu}{D} \in H^{s,\tilde{p}}\end{equation} and 
\begin{equation}\label{def:kcond}\norm{kf}_{H^{s,\tilde{p}}} \leq C \norm{f}_{H^{s,p}}\end{equation}
for some $1<\tilde{p}<2$ satisfying
\begin{equation}\label{def:ptilde}\frac{2}{n+1}+\frac{1}{p} \leq \frac{1}{\tilde{p}} < \frac{2}{n} + \min\left\{\frac{1}{p},\frac{s}{n}\right\}.\end{equation} 
Here, $H^{s,p}$ is the Bessel potential space, defined for a given $s\in\mathbb{R}$ and $p\geq 1$ by
\[H^{s,p}:= \{f\in L^p(\mathbb{R}^n) : \mathcal{F}^{-1}((1+|\xi|^2)^{s/2}\mathcal{F}f) \in L^p(\mathbb{R}^n)\},\] 
where $\mathcal{F}$ and $\mathcal{F}^{-1}$ denote the Fourier transform and its inverse respectively. 

Examples of choices for $s,p,\tilde{p}$ include $1<\tilde{p}<2$ with $s=0$ and $p>3$, or $s=1$ and $p>\frac{6}{5}$ in $\mathbb{R}^2$, or $s=0$, $p>6$, $\frac{3}{2}<\tilde{p}<2$ or $s=1$, $p>2$, $1<\tilde{p}<2$ in $\mathbb{R}^3$, as given in \cite[p.~419]{CakoniXiao2021CommPDE-CornerScatteringAnisotropic}.

Then, we define the admissibility sets for $q$ as follows:
\begin{definition}
    Let $\Omega$ be an open bounded set in $\mathbb{R}^n$, $n=2,3$. We say that $q\in L^2(\Omega)$ is admissible and write $q\in \mathcal{C}$ if $q$ is of the form $q=\varphi\chi_{\omega}$ for the convex polygon or polyhedron $\omega\Subset\Omega$ with corners $x_c$ such that $q$ is $C^\gamma$ H\"older continuous for some $\gamma\in(0,1)$.
\end{definition}

Another structure we are considering is the corona shape.

\begin{definition}\label{def:Corona}
    Let $\tilde{\omega}$ be a convex bounded Lipschitz domain with a connected complement $\Omega\backslash\overline{\tilde{\omega}}$. We say that the set $\omega\Subset\Omega$ belongs to the class of \emph{corona shapes}, denoted by $\mathcal{D}$, if there exist finitely many protruding strictly convex conic cones or polyhedrons $\mathcal{S}^j$, $j=1,2,\dots,\ell$, $\ell\in\mathbb{N}$, i.e. the apexes $x_c^j\in \mathbb{R}^n\backslash\overline{\tilde{\omega}}$ and 
    \[\partial (\mathcal{S}^j\backslash\overline{\tilde{\omega}}) \backslash \partial \mathcal{S}^j \subset \partial \overline{\tilde{\omega}}\quad \text{ and }\quad \cap_{j=1}^\ell \partial (\mathcal{S}^j\backslash\overline{\tilde{\omega}}) \backslash \partial \mathcal{S}^j = \emptyset.\] An example of such a corona shape is in Figure \ref{fig:Corona}.
\end{definition}
\begin{figure}
    \centering
    \includegraphics[width=0.25\textwidth]{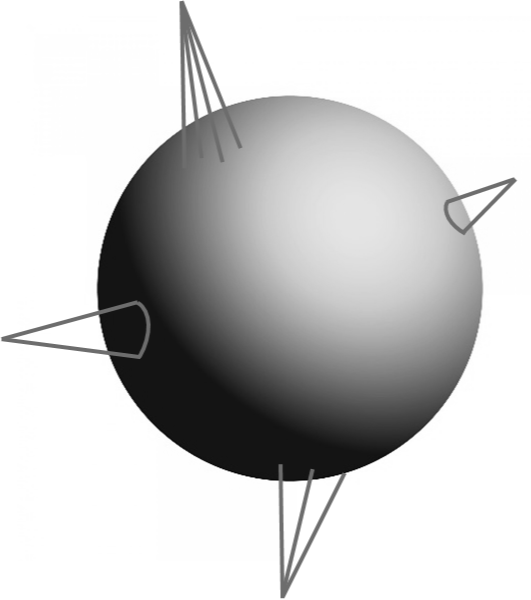}
    \caption{An example of a corona shape}
    \label{fig:Corona}
\end{figure}

\begin{definition}
    Let $\Omega$ be an open bounded set in $\mathbb{R}^n$, $n=2,3$. We say that $q\in L^2(\Omega)$ is admissible and write $q\in \mathcal{E}$ if $q$ is $C^\gamma$ H\"older continuous for some $\gamma\in(0,1)$ and of the form $q=\varphi\chi_{\omega}$ for the corona shape $\omega\in\mathcal{D}$ with corners $x_c$ (see Figure \ref{fig:CornerEmbed}) such that $\omega\Subset\Omega$ and 
    \[\omega = \bigcup_{j=1}^\ell \mathcal{S}^j \cup \tilde{\omega}\] for $\tilde{\omega}$ fixed and $\mathcal{S}^j$ uniquely determined by its apex $x_c$ (i.e. the opening angle $\theta_c$ or edges $e_1,\dots,e_\ell$ depend only on the apex $x_c$).
\end{definition}

For $q$ in $\mathcal{C}$ or $\mathcal{E}$, we have the following main results.

\begin{theorem}\label{thm:mainPolygon}
    Let $\Omega$ be an open bounded set in $\mathbb{R}^n$, $n=2,3$. Consider the BLT problem \eqref{eq:main}--\eqref{eq:main1} satisfying \eqref{def:k} and \eqref{def:kcond}, and suppose that $q\in\mathcal{C}$ is a solution. Then $q$ is uniquely determined by a single boundary measurement $g$, in the sense that the polyhedron $\omega$ (see Figure \ref{fig:CornerSingle}(a)) and the light intensity $\varphi(x_c)$ are uniquely determined for the corners $x_c$ of $\omega$.
\end{theorem}

\begin{theorem}\label{thm:mainCorona}
    Let $\Omega$ be an open bounded set in $\mathbb{R}^n$, $n=2,3$. Consider the BLT problem \eqref{eq:main}--\eqref{eq:main1} satisfying \eqref{def:k} and \eqref{def:kcond}, and suppose that $q\in\mathcal{E}$ is a solution. Then $q$ is uniquely determined by a single boundary measurement $g$, in the sense that the corona shape $\omega$ (see Figure \ref{fig:CornerSingle}(b)) and the light intensity $\varphi(x_c)$ are uniquely determined for the corners $x_c$ of $\omega$.
\end{theorem}
\begin{figure}
    \centering
    \includegraphics[width=0.28\textwidth]{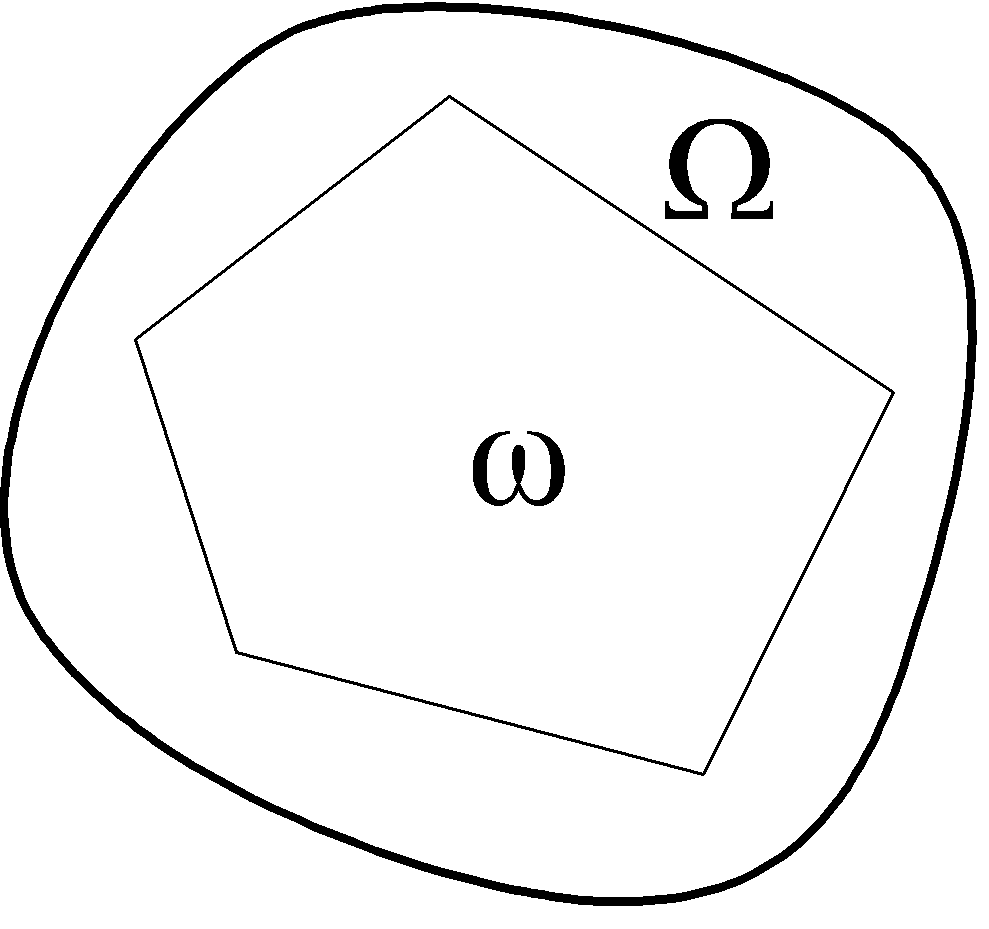}\hspace{2em}\includegraphics[width=0.28\textwidth]{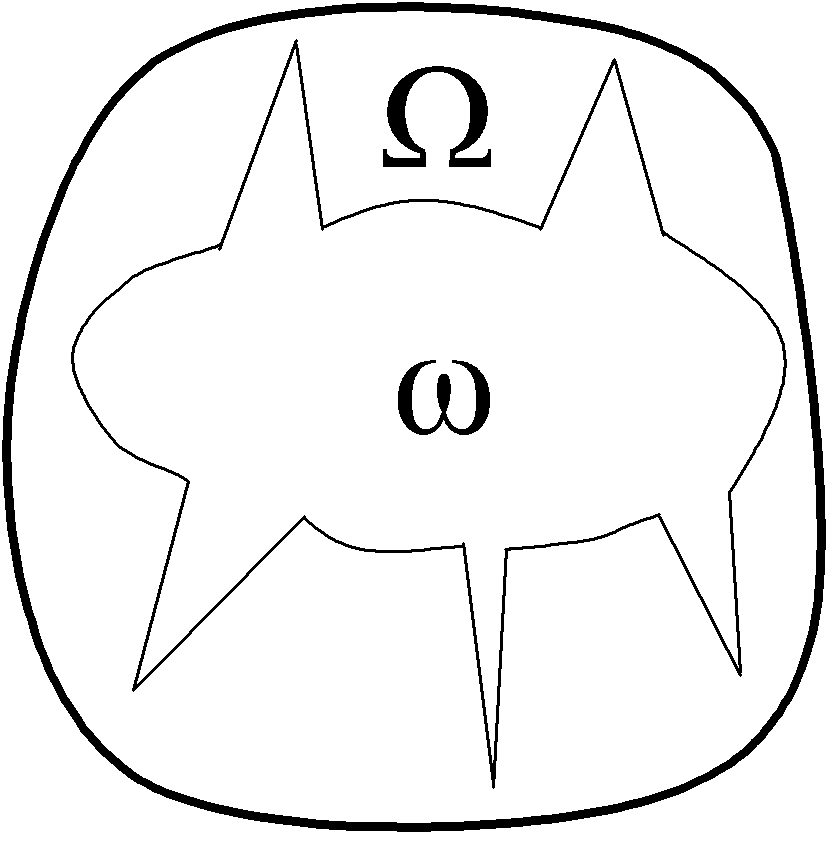}
    \caption{(a) Admissible polygons, (b) Admissible corona shapes}
    \label{fig:CornerSingle}
\end{figure}

\begin{remark}
    We remark that in this case, it does not make sense to consider multiple components, since we are uniquely determining $\omega$ by its corners $x_c$. In the case of the convex polygon/polyhedron, i.e. when $\omega\in\mathcal{C}$, $\omega$ is then given by joining up the corners $x_c$, which will be uniquely determined, using edges, such that $\omega$ is convex. In the case of the corona shape, i.e. when $\omega\in\mathcal{E}$, the convex subset $\tilde{\omega}\subset\omega$ is fixed and only the protruding convex cones or polyhedrons $\mathcal{S}^j$ will be determined.
\end{remark}

Correspondingly, we also have the uniqueness result in the case of embedded polyhedrons or coronas.
\begin{definition}
    We say that $\omega\in\mathcal{F}$ has a \emph{polygonal-nest or polyhedral-nest or corona-nest partition} if there exists 
    \[\omega_N\Subset \omega_{N-1} \Subset \cdots \Subset \omega_2 \Subset \omega_1 = \omega \Subset \Omega \subset \mathbb{R}^n,\] where $\omega_\ell$, $\ell = 1, 2, \dots, N$, $N\in\mathbb{N}$, is such that each $\omega_\ell$ is an open convex simply-connected polygon/polyhedron or corona. An example of a polygonal-nest or polyhedral-nest partition is given in Figure \ref{fig:CornerEmbed}(a), and an example of a corona-nest partition is given in Figure \ref{fig:CornerEmbed}(b).
\end{definition}
    
\begin{figure}
    \centering
    \includegraphics[width=0.28\textwidth]{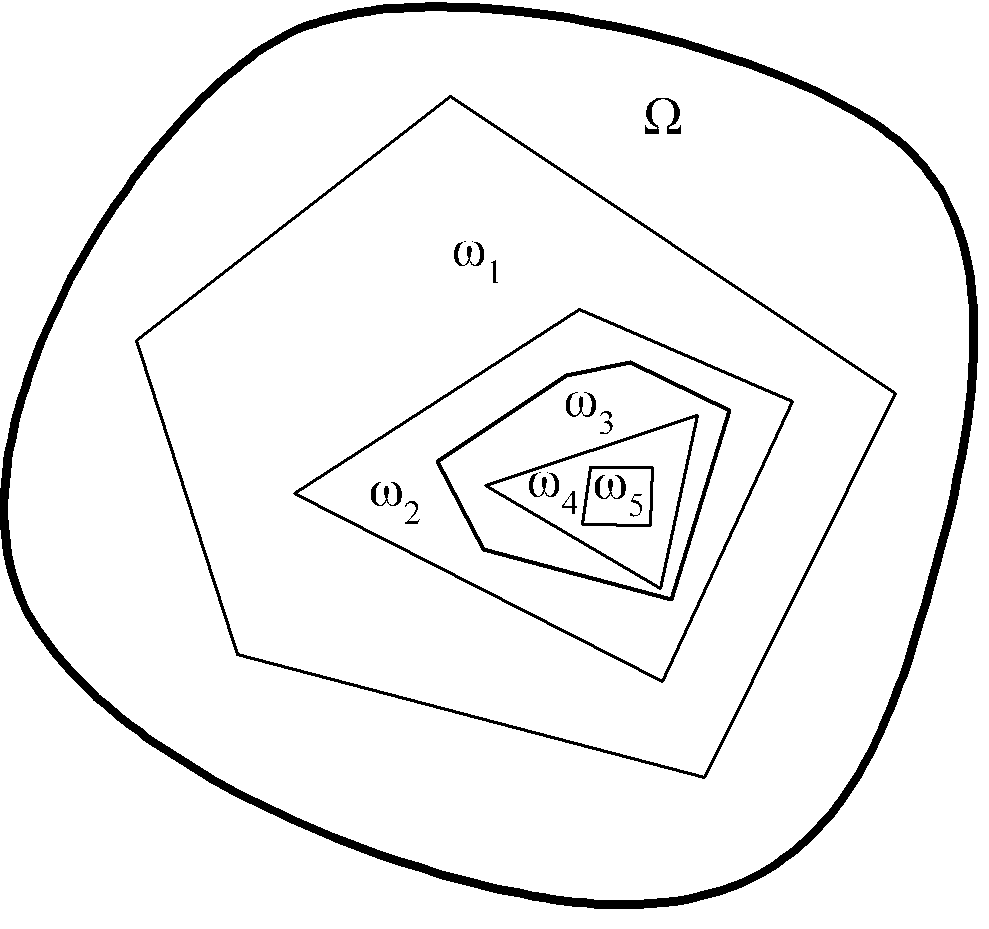}\hspace{1em}\includegraphics[width=0.28\textwidth]{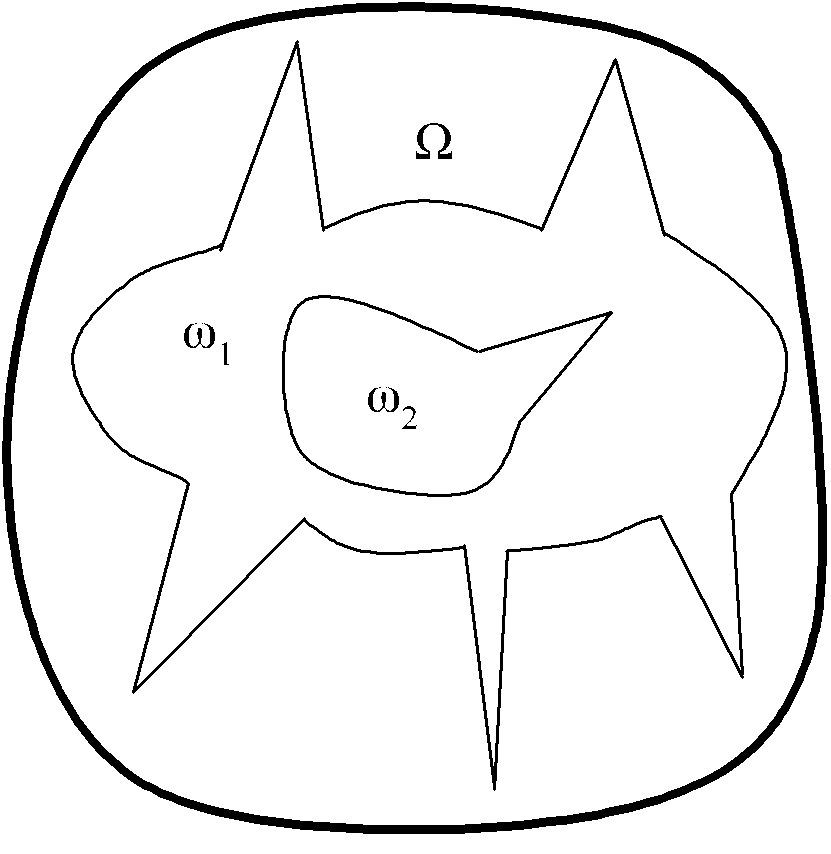}
    \caption{(a) Admissible polygonal-nest or polyhedral-nest partition, (b) Admissible corona-nest partition}
    \label{fig:CornerEmbed}
\end{figure}

The admissibility condition on $q$ is subsequently defined as follows.
\begin{definition}
    Let $\Omega$ be an open bounded set in $\mathbb{R}^n$, $n=2,3$. We say that $q$ is admissible and write $q\in \mathcal{G}$ if $q$ is of the form $q=\sum_{j=1}^m \varphi_j\chi_{\omega_j}$ for the partition $\omega\in\mathcal{F}$ (see Figure \ref{fig:CornerEmbed}) with 
    \[\omega_m\Subset \omega_{m-1} \Subset \cdots \Subset \omega_2 \Subset \omega_1 = \omega, \] such that $w_j\in\mathcal{C}$ or $\mathcal{E}$ with corners $x_c$ and $\varphi_j$ is constant for each $j$. Without loss of generality, we assume that $\varphi_1\neq0$ and $\varphi_j\neq \varphi_{j+1}$ for each $j$.
\end{definition}

As a corollary, we have the following result

\begin{corollary}\label{cor:PolyEmbed}
Let $\Omega$ be an open bounded set in $\mathbb{R}^n$, $n=2,3$. Consider the BLT problem \eqref{eq:main}--\eqref{eq:main1} satisfying \eqref{def:k} and \eqref{def:kcond}, and suppose that $q\in\mathcal{G}$ is a solution. Then $q$ is uniquely determined by a single boundary measurement $g$, in the sense that the polyhedral or corona shape domains $\omega_j$ and the constant light intensities $\varphi_j$ are uniquely determined for all $j=1,\dots,m$.
\end{corollary}

\subsection{Geometrical Setup}

For $n=2,3$, consider the convex conic cone $\mathcal{C}_{x_c,\theta_c}\subset\Omega$ with apex $x_c$ and axis $v_c$ and opening angle $2\theta_c\in(0,\pi)$, defined by
\[\mathcal{C}_{x_c,\theta_c}:=\{y\in\Omega:0\leq\angle (y-x_c,v_c)\leq\theta_c,\theta_c\in(0,\pi/2)\}.\] Define the truncated conic cone by
\[\mathcal{C}^h_{x_c,\theta_c}:=\mathcal{C}_{x_c,\theta_c}\cap B_h,\] where $B_h:=B_h(x_c)$ is an open ball contained in $\Omega$ centred at $x_c$ with radius $h>0$. Observe that both $\mathcal{C}_{x_c,\theta_c}$ and $\mathcal{C}^h_{x_c,\theta_c}$ are Lipschitz domains. 

We also introduce the polyhedral corner in $\mathbb{R}^3$ as follows: Let $\mathcal{K}_{x_c;e_1,\dots,e_\ell}$ be a polyhedral cone with apex $x_c$ and edges $e_i$, $i=1,\dots,\ell$, $\ell\geq3$, where $e_i$ are mutually linearly independent vectors in $\mathbb{R}^3$. Assume that $\mathcal{K}_{x_c;e_1,\dots,e_\ell}$ is strictly convex, so that it can be fitted into a conic cone $\mathcal{C}_{x_c,\theta_c}$. Given a constant $h>0$, define the truncated polyhedral corner 
\[\mathcal{K}_{x_c;e_1,\dots,e_\ell}^h := \mathcal{K}_{x_c;e_1,\dots,e_\ell} \cap B_h.\]

Throughout the rest of the paper, the results hold for both these two types of corners, and we denote 
\[\mathcal{S}_h: = \mathcal{C}_{x_c,\theta_c} \text{ or }\mathcal{K}_{x_c;e_1,\dots,e_\ell}\]
 
We also recall some asymptotics for a CGO solution we will be using. 
\begin{lemma}\label{CGOAsympLemma}
Suppose $p$ is such that $p>1+\frac{2}{n-1}$ and $\frac{n}{p}<\frac{2}{n+1}+s$. Let $\tau\in\mathbb{R}^+$. Then there exists a solution $w$ to the equation  
\begin{equation}\label{CGOEq}- \nabla \cdot (D\nabla w(x)) + \mu w(x) = 0\quad \text{ in }\mathbb{R}^n,\end{equation}
of the form 
\begin{equation}\label{CGO}w = \frac{1}{\sqrt{D}} e^{-\tau (\xi + i\xi^\perp)\cdot (x-x_c)}(1+r(x))\end{equation} 
such that $\xi\cdot\xi^\perp=0$, $\xi,\xi^\perp\in\mathbb{S}^{n-1}$, and 
\begin{equation}\label{CGOSolnEst}\norm{r}_{H^{s,p}} = \mathcal{O}(\tau^{n(1/\tilde{p}-1/p)-2}),\end{equation} for $\tilde{p}$ defined in \eqref{def:ptilde}. 

Furthermore, there exists a positive number $\rho$ depending on the truncated (with constant $h$) cone $\mathcal{S}_h$ satisfying 
\begin{equation}\label{CGOCond} \xi\cdot\widehat{(x-x_c)}\geq \rho > 0\quad\text{ for all } x\in \mathcal{S}_h,\end{equation} where $\hat{x} = \frac{x}{|x|}$. Moreover, for sufficiently large $\tau$, 
\begin{equation}\label{CGOEst0}\left|\int_{\mathcal{S}_h}w\right|\geq \frac{C_{\mathcal{S}_h}}{\tau^{n-1}} -\mathcal{O}\left(\frac{1}{\tau}e^{-\frac{1}{2}\rho h \tau}\right),\end{equation}
\begin{equation}\label{CGOEst1}\norm{w}_{L^2(\mathcal{S}_h)} \lesssim \left(1+\tau^{-\frac{2}{3}}\right)e^{-\rho h \tau} \text{ in }\mathbb{R}^2, \quad \norm{w}_{L^2(\mathcal{S}_h)} \lesssim \left(1+\tau^{-\frac{2}{5}}\right)e^{-\rho h \tau} \text{ in }\mathbb{R}^3,\end{equation}
\begin{equation}\label{CGOEst2}\norm{\nabla w}_{L^2(\mathcal{S}_h)} \lesssim (1+\tau)\left(1+\tau^{-\frac{2}{3}}\right)e^{-\rho h \tau} \text{ in }\mathbb{R}^2, \quad\norm{\nabla w}_{L^2(\mathcal{S}_h)} \lesssim (1+\tau)\left(1+\tau^{-\frac{2}{5}}\right)e^{-\rho h \tau} \text{ in }\mathbb{R}^3,\end{equation}
and 
\begin{equation}\label{CGOEst3}\begin{aligned}\left|\int_{\mathcal{S}_h}|x|^\alpha w\right|\lesssim \tau^{-(\alpha+\frac{29}{12})}+\tau^{-(\alpha+2)}+\frac{1}{\tau}e^{-\frac{1}{2}\rho h \tau} \text{ in }\mathbb{R}^2, \\ \left|\int_{\mathcal{S}_h}|x|^\alpha w\right|\lesssim \tau^{-(\alpha+\frac{121}{40})}+\tau^{-(\alpha+3)}+\frac{1}{\tau}e^{-\frac{1}{2}\rho h \tau} \text{ in }\mathbb{R}^3\end{aligned}\end{equation}
for all $0<\alpha<1$.
Here, we use the symbol `` $\lesssim$" to denote that the inequality holds up to a constant which is independent of $\tau$.
\end{lemma}

\begin{proof}
    The proof follows much of that given in \cite{CakoniXiao2021CommPDE-CornerScatteringAnisotropic} and \cite{DiaoFeiLiu-2022-EM-Shape-Corners2}, and we give a sketch here. We first observe that \eqref{CGOEq} reduces to 
    \[\Delta v + kv = 0 \quad \text{ in }\mathbb{R}^n\]
    using the transformation 
    \[v = \sqrt{D} w, \quad k = -\frac{\Delta \sqrt{D}}{\sqrt{D}} - \frac{\mu}{D},\] as described in \cite{BalUhlmann2010IP-EllipticIP}. Then, by Proposition 3.1 of \cite{CakoniXiao2021CommPDE-CornerScatteringAnisotropic}, we have that $w$ of the form \eqref{CGO} solves \eqref{CGOEq} with the residual $r$ satisfying the estimate \eqref{CGOSolnEst}. 

    Next, by the convexity of $\mathcal{C}_{x_c,\theta_c}$ or $\mathcal{K}_{x_c;e_1,\dots,e_\ell}$, there exists a positive constant $\rho$ such that \eqref{CGOCond} holds for an open set of $\{\widehat{x-x_c}\in\mathbb{S}^{n-1}\}$. 

    Then, \eqref{CGOEst0} follows from Proposition 3.1 and Proposition 4.1 of \cite{DiaoFeiLiu-2022-EM-Shape-Corners2}, for a positive constant $C_{\mathcal{S}_h}$ depending only on $\mathcal{S}_h$ and $\rho$. Finally, \eqref{CGOEst1}--\eqref{CGOEst3} follow from Lemmas 3.4 and 4.2 of \cite{DiaoFeiLiu-2022-EM-Shape-Corners2} by choosing $s=1$, $p=\frac{1}{8}$, and for $\tilde{p}$ defined by \eqref{def:ptilde}, $\tilde{p}=\frac{48}{38}$ and $\tilde{p}=\frac{120}{79}$ in $\mathbb{R}^2$ and $\mathbb{R}^3$ in \eqref{CGOSolnEst} respectively.
\end{proof}

\begin{remark}
    We remark that the assumptions on $k$ in \eqref{def:k}--\eqref{def:ptilde} and Lemma \ref{CGOAsympLemma} are clearly satisfied when $D$ and $\mu$ are piecewise constant, as in the case considered in \cite{BLT2003}.
\end{remark}

\subsection{Proof of Main Results}

We first conduct a microlocal analysis on the corners of $\omega$. We begin with an auxiliary result.

\begin{theorem}\label{AuxThm}
    Let $\mathcal{S}_h$ be a corner. For $C^\gamma$ H\"older continuous functions $q,Q$, consider the following system of differential equations for $u\in H^2_{loc}(\mathcal{S}_h)$ and $v\in H^2_{loc}(\mathcal{S}_h)$:
    \[\begin{cases}
    -\nabla \cdot (D\nabla u) + \mu u = q &\quad \text{ in } \mathcal{S}_h, \\
    -\nabla \cdot (D\nabla v) + \mu v = Q &\quad \text{ in } \mathcal{S}_h, \\
    u = v, \quad \partial_\nu u = \partial_\nu v &\quad \text{ on }\partial \mathcal{S}_h \backslash \partial B_h.\end{cases}
    \]
     Then one has
    \[q(x_c) = Q(x_c),\]
    where $x_c$ is the apex of $\mathcal{S}_h$.
\end{theorem}

\begin{proof}
    Taking the difference of the equations and multiplying by the solution $w$ to \eqref{CGOEq} given by \eqref{CGO} and integrating in the truncated cone $\mathcal{S}_h$, we have, by Green's formula, 
\begin{equation}\label{DiffEq}
    \int_{\mathcal{S}_h} - (q-Q)  w(x) \,dx = - \int_{\partial \mathcal{S}_h} w(x) \partial_\nu \tilde{u}(x) \, d\sigma + \int_{\partial \mathcal{S}_h} \tilde{u}(x) \partial_\nu w(x) \, d\sigma.
\end{equation}
Here, we write $\tilde{u}=u-v$.

By the H\"older continuity of $q$ and $Q$, we can expand $F:=-q+Q$ as follows: 
\[F=F(x_c)+\delta F,\quad |\delta F|\leq \norm{F}_{C^\gamma(\mathcal{S}_h)}|x|^\gamma,\] 
where \[F(x_c)=-q(x_c) +Q(x_c).\] Therefore, the left-hand-side of \eqref{DiffEq} can be expanded as
\begin{align*}
    \int_{\mathcal{S}_h}F w(x) \,dx = F(x_c)\int_{\mathcal{S}_h} w(x) \,dx + \int_{\mathcal{S}_h}\delta F w \,dx,
\end{align*}
where
\begin{equation}\label{LHSEst}
    \left|\int_{\mathcal{S}_h}\delta F w \,dx\right| \leq \norm{F}_{C^\gamma(\mathcal{S}_h)} \int_{\mathcal{S}_h}|x|^\gamma |w| \,dx \leq C \int_{\mathcal{S}_h}|x|^\gamma |w| \,dx 
\end{equation}
for some constant $C$. 

On the other hand, the right-hand-side of \eqref{DiffEq} can be analysed as follows: By the Cauchy-Schwarz inequality and the trace theorem, 
\begin{align*}
    \left|\int_{\partial \mathcal{S}_h} w(x) \partial_\nu \tilde{u}(x) \, d\sigma\right| & \leq \norm{w}_{H^{\frac12}(\partial \mathcal{S}_h)}\norm{\partial_\nu \tilde{u}}_{H^{-\frac12}(\partial \mathcal{S}_h)} \\
    & \leq C \norm{w}_{H^1(\partial \mathcal{S}_h)}\norm{\tilde{u}}_{H^1(\mathcal{S}_h)}
    \\
    & \lesssim (1+\tau)\left(1+\tau^{-d}\right)e^{-\rho h\tau},
\end{align*}
by \eqref{CGOEst1} and \eqref{CGOEst2}, while
\begin{align*}
    \left|\int_{\partial \mathcal{S}_h} \tilde{u}(x) \partial_\nu w(x) \, d\sigma\right| & \leq \norm{\tilde{u}}_{L^2(\partial \mathcal{S}_h)}\norm{\partial_\nu w}_{L^2(\partial \mathcal{S}_h)} \\
    & \leq C \norm{\tilde{u}}_{H^1(\partial \mathcal{S}_h)}\norm{\partial_\nu w}_{L^2(\partial\mathcal{S}_h)}
    \\
    & \lesssim (1+\tau)\left(1+\tau^{-d}\right)e^{-\rho h\tau}.
\end{align*}
by \eqref{CGOEst2}, where $d = \frac{2}{3}$ for $n=2$ and $d=\frac{2}{5}$ for $n=3$. Note that the (trace or Sobolev) constants $C>0$ here may be different, and different from that of \eqref{LHSEst}.

Combining these two estimates with \eqref{LHSEst} and \eqref{CGOEst0}, then multiplying by $\tau^{n-1}$ on both sides and letting $\tau\to\infty$, we have that $F(x_c)=0$, i.e. $q(x_c) = Q(x_c)$.

\end{proof}

\begin{proof}[Proof of Theorem \ref{thm:mainPolygon}]
    Suppose on the contrary that there exists two solutions $q,\hat{q}\in\mathcal{C}$, of the form $q=\varphi\chi_{\omega}$ and $\hat{q}=\hat{\varphi}\chi_{\hat{\omega}}$, to the BLT problem. We will first show that $\omega\Delta\hat{\omega} := (\omega\backslash\hat{\omega}) \cup (\hat{\omega}\backslash\omega)$ cannot possess a corner on the connected component $\partial(\Omega\backslash\overline{\omega\cup\hat{\omega}})$ that connects to $\partial\Omega$, if $(\omega,q),(\hat{\omega},\hat{q})$ satisfy the BLT problem. Indeed, since $q$, $\hat{q}\in\mathcal{C}$ satisfy the BLT problem corresponding to $u,\hat{u}\in H^2_{loc}(\Omega)$ respectively, by unique continuation principle, it holds that
    \[\begin{cases}
    -\nabla \cdot (D\nabla u) + \mu u = q &\quad \text{ in } \mathcal{S}_h, \\
    -\nabla \cdot (D\nabla \hat{u}) + \mu \hat{u} = \hat{q} &\quad \text{ in } \mathcal{S}_h, \\
    u = \hat{u}, \quad \partial_\nu u = \partial_\nu \hat{u} &\quad \text{ on }\partial \mathcal{S}_h \backslash \partial B_h.\end{cases}
    \] 
    Since $q,\hat{q}$ are H\"older continuous, by Theorem \ref{AuxThm}, we have that $q(x_c)=\hat{q}(x_c)$, i.e. \begin{equation}\label{MainThmPolygonResult}\varphi(x_c)\chi_{\omega}(x_c)=\hat{\varphi}(x_c)\chi_{\hat{\omega}}(x_c).\end{equation} Therefore, $\omega\Delta\hat{\omega}$ cannot possess a corner. 
    Since $\omega\Delta\hat{\omega}$ cannot possess a corner, and $\omega,\hat{\omega}$ are polyhedrons, this means that $\omega=\hat{\omega}$, and consequently, \eqref{MainThmPolygonResult} also implies $\varphi(x_c)=\hat{\varphi}(x_c)$ for every corner $x_c$ of $\omega$.
   
\end{proof}

\begin{proof}[Proof of Theorem \ref{thm:mainCorona}]
    This follows easily as in the polyhedral case by contradiction. Suppose on the contrary that $q,\hat{q}\in\mathcal{E}$ with corona shapes $\omega\neq\hat{\omega}$, $\omega,\hat{\omega}\in\mathcal{D}$. Then, by the assumptions on $\omega$ and $\hat{\omega}$, there exists a corner $\mathcal{S}_h\subset\hat{\omega}\backslash\omega$. But $q$ and $\hat{q}$ satisfy the BLT problem, so as in the previous proof, $\omega\Delta\hat{\omega}$ cannot possess a corner, and we arrive at a contradiction so $\omega = \hat{\omega}$. Finally, invoking Theorem \ref{AuxThm} again since $q$ and $\hat{q}$ are H\"older continuous, we obtain $\varphi(x_c)=\hat{\varphi}(x_c)$ for every corner $x_c$ of $\omega$.
\end{proof}

Correspondingly, we can prove Corollary \ref{cor:PolyEmbed}.

\begin{proof}[Proof of Corollary \ref{cor:PolyEmbed}]
    This follows similarly to the case of $C^2$ domains in Corollary \ref{cor:SmoothEmbed}. Assume on the contrary that $q,\hat{q}\in\mathcal{G}$ satisfy the BLT problem with the given assumptions, with the form $q=\sum_{j=1}^m \varphi_j\chi_{\omega_j}$ and $\hat{q}=\sum_{j=1}^M\hat{\varphi}_j\chi_{\hat{\omega}_j}$. Then, since $q,\hat{q}\in L^2(\Omega)$ with $q|{\partial\omega_1}=\varphi_1,\hat{q}|{\partial\omega_1}=\hat{\varphi}_1$ such that $\varphi_1,\hat{\varphi}_1\neq0$ by assumption, using Theorem \ref{thm:mainPolygon} or Theorem \ref{thm:mainCorona}, we have that $\omega_1 = \hat{\omega}_1$ and $\varphi_1 = \hat{\varphi}_1\neq0$ on the corners $x_c$ of $\omega_1$. Since $\varphi_1,\hat{\varphi}_1$ are constants by the assumption $q\in\mathcal{G}$, we have that $\varphi_1 = \hat{\varphi}_1\neq0$ on $\omega_1\backslash\overline{\omega}_2$.

    We will show that $\omega_j=\hat{\omega}_j$ and $\varphi_j=\hat{\varphi}_j$ for every $j=1,\dots,m$ by induction. Suppose this holds for $j=1,\dots,\ell$. Then, since $\omega_{\ell+1}\Subset\omega_\ell$, given any open covering of $\omega_{\ell+1}$ by open sets of $\omega_\ell$, there is a finite subcollection covering $\omega_{\ell+1}$. Therefore, for each corner $x_c$, we can find $h_{\ell+1}$ with the open ball $B_{h_{\ell+1}}(x_c)$ such that the unique continuation principle holds in $B_{h_{\ell+1}}(x_c)$. Then, applying the unique continuation principle, we have that 
    \[\begin{cases}
    -\nabla \cdot (D\nabla u) + \mu u = \varphi_{\ell+1}\chi_{\omega_{\ell+1}} &\quad \text{ in } \mathcal{S}_{h_{\ell+1}}, \\
    -\nabla \cdot (D\nabla \hat{u}) + \mu \hat{u} = \hat{\varphi}_{\ell+1}\chi_{\hat{\omega}_{\ell+1}} &\quad \text{ in } \mathcal{S}_{h_{\ell+1}}, \\
    u = \hat{u}, \quad \partial_\nu u = \partial_\nu \hat{u} &\quad \text{ on }\partial \mathcal{S}_{h_{\ell+1}} \backslash \partial B_{h_{\ell+1}}.\end{cases}
    \] 
    Applying Theorem \ref{AuxThm} then gives 
    \[\omega_{\ell+1} = \hat{\omega}_{\ell+1} \quad \text{ and } \quad \varphi_{\ell+1}(x_c) = \hat{\varphi}_{\ell+1}(x_c).\] Since $\varphi_{\ell+1},\hat{\varphi}_{\ell+1}$ are constants, this holds in $\omega_{\ell+1}\backslash\overline{\omega}_{\ell+2}$. Repeating inductively, we have the result for $j=1,\dots$ and $m=M$.
\end{proof}

\section{Inversion algorithm}\label{sec:algo}
In this section, we discuss numerical algorithms designed to solve inverse problems. As we are aware, inverse problems are inherently ill-posedness. One approach to mitigate the ill-posedness of these problems is to employ the Tikhonov regularisation method.
The essence of the Tikhonov regularisation method expands  the cost functional with a quadratic term, i.e.,
\begin{align}
\label{variational reg}
J[q]:=\frac{1}{2}\parallel {F}(q)-\Phi^{\delta}\parallel_{L^2}^2+\frac{\lambda}{2}\parallel q \parallel_{L^2}^2,
\end{align}
where $F(q)$ represents the  measurement map of the \eqref{eq:mea}, $\Phi^{\delta}$ signifies the  measurement data and $\lambda$ is the regularization parameter. Moreover, the measurement data $\Phi^{\delta}$ and the exact data $\Phi$ satisfies $\|\Phi^{\delta}-\Phi\|_{L^2}\leq\delta$. Here $\delta$ denotes the noise level.

In this paper, we utilise the Levenberg-Marquardt (LM) method \cite{doicu2010numerical,hanke1997regularizing} to determine the minimizer of $J[q]$, which is essentially a modified version of the Gauss-Newton iteration. Assuming that $\tilde{q}$ is an estimate of the true solution ${q}^{\dag}$, we can approximately substitute the nonlinear mapping ${F}$ in (\ref{variational reg}) with its linearization around $\tilde{q}$. Consequently, the minimization of (\ref{variational reg}) can be interpreted as an attempt to minimize
\begin{align}
\label{var-reg}
J[\delta q]:=\frac{1}{2}\parallel {F}'(\tilde{q})\delta q_0-(\Phi^{\delta}-{F}(\tilde{q}))\parallel_{L^2}^2+\frac{\lambda}{2}\parallel \delta q \parallel_{L^2}^2,
\end{align}
where $\delta q=q-\tilde{q}$. Subsequently, we present a detailed outline of the sequential steps involved in the Levenberg-Marquardt algorithm, see algorithm \ref{algorithm-pgddf}.
\begin{algorithm}
\caption{LM method for solving the variational problem (\ref{var-reg}).}
      1:~~Choose $q^0, \lambda$, and set $i=0$;\\
      ~2:~~Solve the direct problem and determine the residual $\textbf{F}_i=\Phi^{\delta}-{F}(q^i)$;\\
      3:~~Compute the Jacobian $G={F}'(q^i)$;\\
      ~4:~~Calculate $\delta q^i=(G^*G+\lambda I)^{-1}(G^* \textbf{F}_i)$, where $G^*$ is the conjugate of $G$;\\
      ~5:~~Update the solution $q^i$ by $q^{i+1}=q^i+\delta q^i$, \\
      6:~~Increase $i$ by one and go to step 2, repeat the above procedure until a stopping criterion \\ $~~~~~$is satisfied.\\
   \label{algorithm-pgddf}
\end{algorithm}

Generally, the choice of optimal regularization parameter is critical when implementing regularization strategies. Specifically, in the context of simultaneous inversion problems, certain inversion algorithms may falter owing to the inherent uncertainty of the inversion problem, especially if empirical selection of regularisation parameters is employed. Consequently, we will leverage the properties of sigmoid-type functions (as referenced in \cite{%kubat1999neural, 
li2013simultaneous}) to determine the regularization parameter for our implementation of the optimal perturbation algorithm, i.e.,
\begin{align*}
\lambda=\lambda(i)=\frac{1}{1+e^{\beta(i+i_0)}},
\end{align*}
where $i$ is the number of iterations, $i_0$ is an a priori chosen number and $\beta > 0$ is the adjust parameter.

We are aware that the function $\lambda(i)$ exhibits a continuous decrease in values, gradually approaching zero as $i$ increases. Furthermore, it experiences a rapid diminishment after $i\geq 10$. Given that the aforementioned properties are essential in general regularization theory for regularization parameters, we  employ numerical inversions to tackle the inverse problem. 

\section{Numerical results and discussion}\label{sec:resu}

In all subsequent calculations, the domain $\Omega$ is considered to be a circle in the case of $n=2$ or a ball in the case of $n=3$, centered at the origin with a radius of 3, denoted as $B_3(0)$. We employ the finite element method to solve the forward problem and the total number of measurements is 200. The noise level $\delta=0.01$.

During the iterative process, the Jacobian matrix $G$ is computed using a finite difference method. The maximum number of iteration steps is set as 20, and the following stopping rule is applied
\begin{align*}
E_i=\parallel q^i-q^{i-1}\parallel\leq 10^{-2}.
\end{align*}
The noisy measured data is generated by
\begin{align*}
\Phi^\delta=\Phi+\delta \Phi (2\textrm{rand}(\textrm{size}(\Phi))-1),
\end{align*}
where $\Phi$ is the exact data, rand  is representative of a uniform distribution within the range of $[0, 1]$.

The accuracy of the approximate solution $\tilde{q}$ by the algorithm \ref{algorithm-pgddf} is characterized by comparing to the exact solution $q^{\dag}$ via the relative error
\begin{align*}
e_r=\frac{\|\tilde{q}-q^{\dag}\|_{L^2}}{\|q^{\dag}\|_{L^2}}.
\end{align*}

\subsection{Smooth Domains} 
In this subsection we consider several numerical examples for smooth domains.

\begin{exm}\label{cir}
 We consider a 2-dimensional scenario where $\omega$ is a circle denoted as $B_1(0)$, and the intensity of the source within $\omega$ is set to $\varphi=1$.
\end{exm}

In the first example \ref{cir}, the distribution $D(x)$ takes the form of a piecewise constant function. More precisely, the material within the $B_2(0)$ section represents the lungs and is characterized by $\mu_a=0.023$ and $\mu_s'=2$. Conversely, the material in the $\Omega\backslash\overline{B_2(0)}$ region corresponds to the muscle, with respective values of $\mu_a=0.007$ and $\mu_s'=1.031$. The function $g^- = 0$ signifies the simulation taking place within a dark environment. In the iterative process, the regularization parameter is set to $i_0=0$ and $\beta=0.7$, the initial  iteration  $q^0$ is chosen for a circle with centers at $(0.3, 0.5)$ and a radius of 0.5, and the intensity of the source within $\omega$ is set to $\varphi=0.8$.

 In Figure \ref{circle}, we analyze the accuracy of the reconstructed solution compared to the exact solution. In Figure \ref{circle} (a)-(d), we present the evolution of the reconstructed solution over different iteration steps $i$. Additionally, in Figure \ref{circle} (f), we illustrate the variation of the relative error $e_r$ of the number of iteration steps. Based on the obtained results, it is evident that the solution $\tilde{q}$ achieved after six iterations closely approximates the exact solution, which further verifies that when $\omega$ is a smooth domain, the location, shape and size of this light source can be effectively reconstructed by boundary measurements \eqref{Neumann}.

\begin{figure}
\centering
\subfloat[$i=0$ ]{\includegraphics[width=0.3\textwidth]
                   {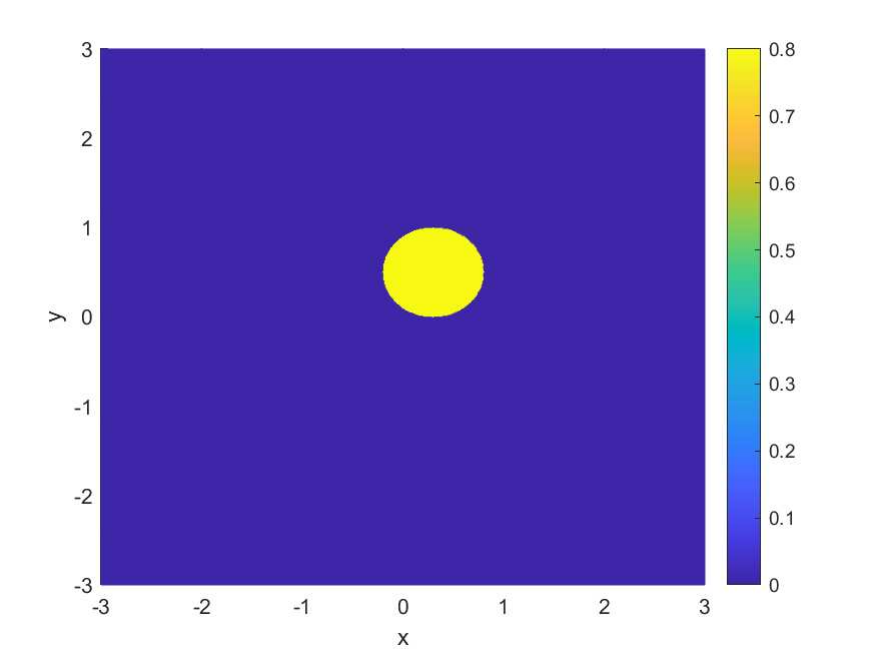}}
\subfloat[$i=1$ ]{\includegraphics[width=0.3\textwidth]
                   {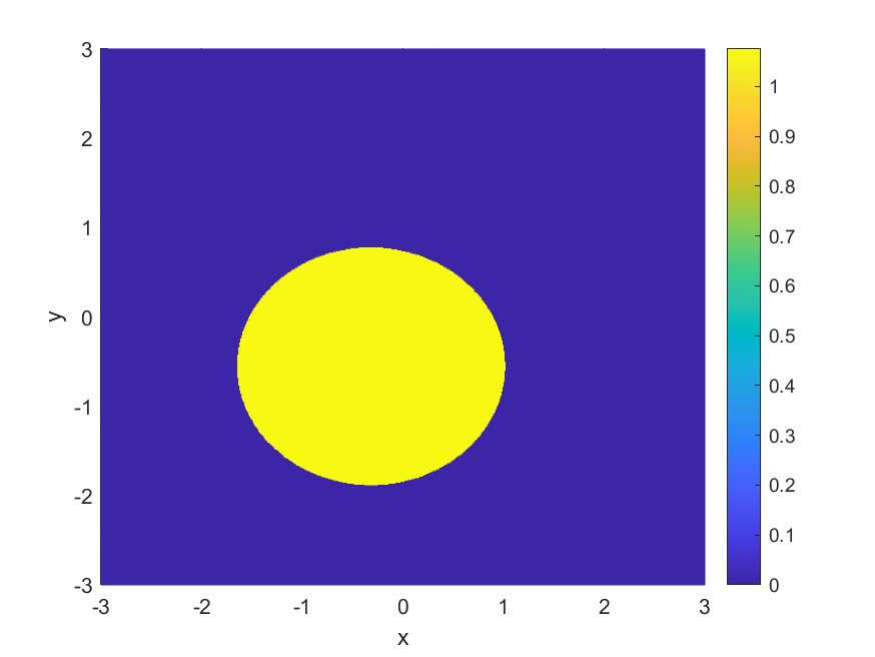}}
\subfloat[ $i=2$]{\includegraphics[width=0.3\textwidth]
                   {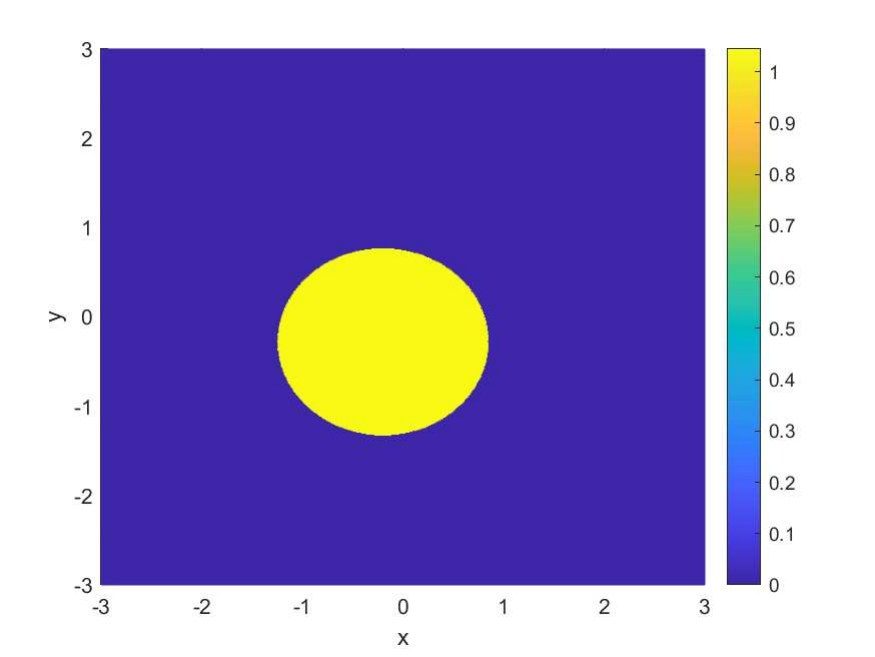}}\\
\subfloat[ $i=6$]{\includegraphics[width=0.3\textwidth]
                   {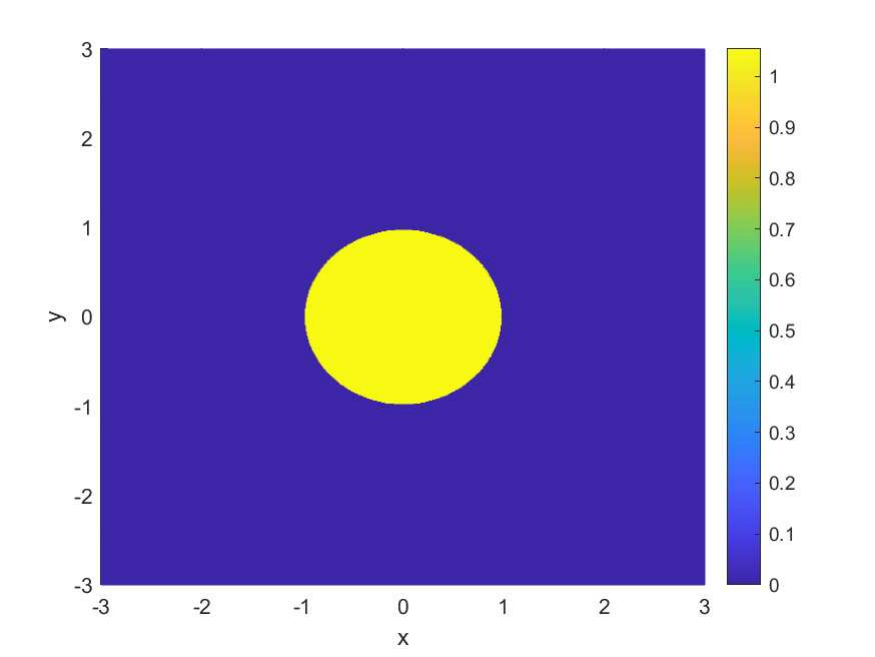}}
\subfloat[ $exact\ solution $]{\includegraphics[width=0.3\textwidth]
                   {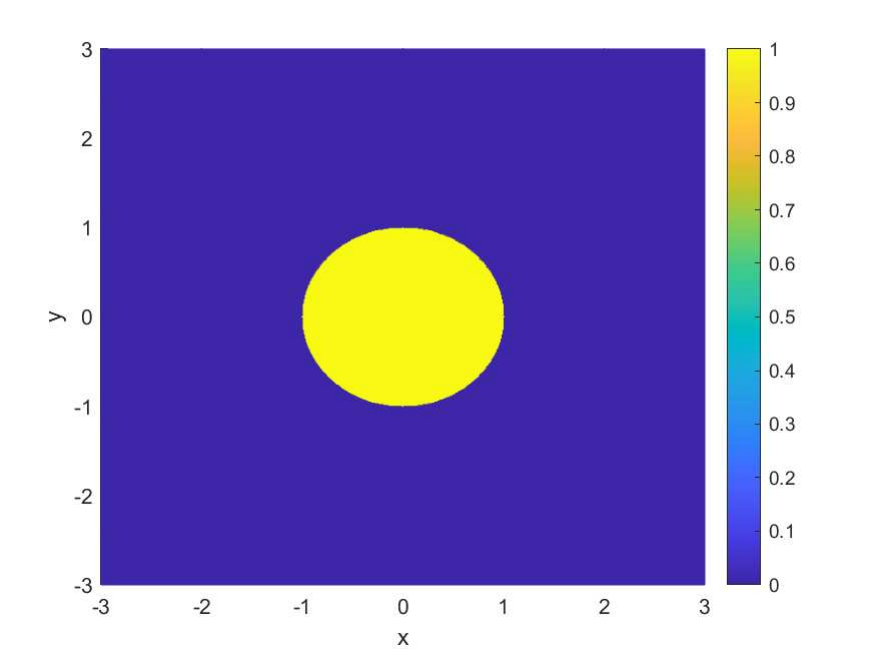}}
\subfloat[ $e_r $]{\includegraphics[width=0.3\textwidth]
                   {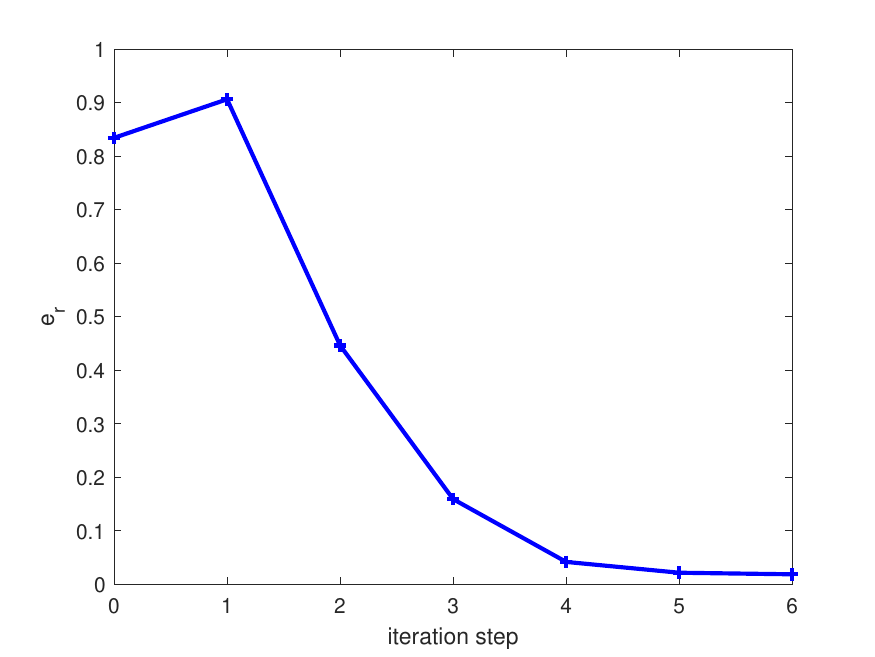}}

\caption{\label{circle}   (a) (b) (c) (d) reconstructed solutions at different iteration steps, (e) exact solution, (f) $e_r$ for different iteration steps $i$. }
\end{figure}

\begin{exm}\label{exm2}
In the second example, we consider the case where $q$ is a nested partition in two dimensions, i.e., $q=\sum_{j=1}^2\varphi_j\chi_{\omega_j}$, where $\omega_1=B_{3/2}(0)\backslash \overline{B_{1/2}(0)} $ and $\omega_2=B_{1/2}(0)$,  and the corresponding source intensities are  $\varphi_1=1$ and $\varphi_2=2$, respectively.
\end{exm}

  For Example \ref{exm2}, the function $D(x)$ is constant within the domain $\Omega$, with the material representing the heart having properties $\mu_a=0.011$ and $\mu_s'=1.096$, and the input function $g^-=0$ in the forward problem. During the iteration, the regularization parameters are chosen with $i_0=0$ and $\beta=0.5$, while the initial guess for the source function is assigned as $q^0=\sum_{j=1}^2\varphi_j\chi_{\omega_j}$,  where $\omega_1=B_{2.1}(0.5,0.5)\backslash \overline{B_{0.4}(0.1,0.1)} $ and $\omega_2=B_{0.4}(0.1,0.1)$,  and the corresponding initial source intensities are  $\varphi_1=\varphi_2=1.8$, respectively.

In Figure \ref{circle_n}, we analyze the accuracy of the reconstructed solution in comparison to the exact solution. Figures \ref{circle_n} (a)-(d) depict the progressive development of the reconstructed solution at different iteration steps $i$. Furthermore, in Figure \ref{circle_n} (f), we present the progression of the relative error across the iteration steps. It is clear that the reconstructed solution is obtained in excellent agreement with the exact solution. It shows that even if the internal source is nested in this case, the exact reconstruction of the source can be achieved by measurements on the boundary.

\begin{figure}
\centering
\subfloat[$i=0$ ]{\includegraphics[width=0.3\textwidth]
                   {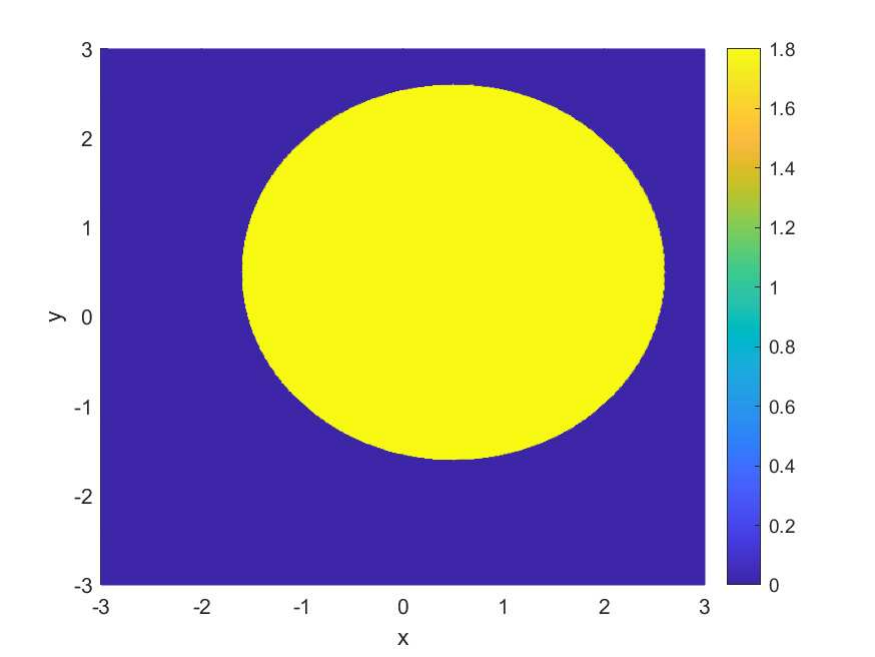}}
\subfloat[$i=1$ ]{\includegraphics[width=0.3\textwidth]
                   {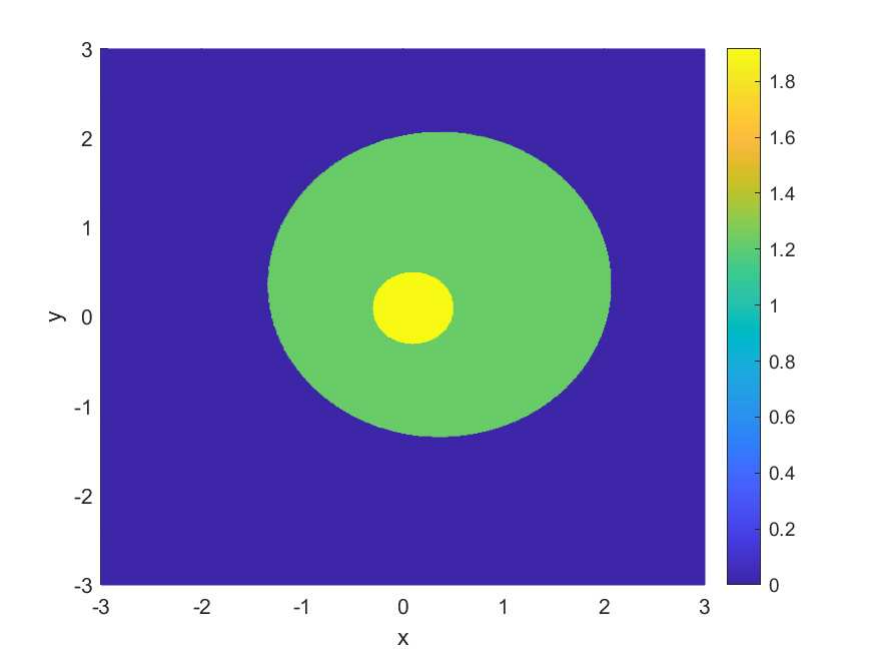}}
\subfloat[ $i=2$]{\includegraphics[width=0.3\textwidth]
                   {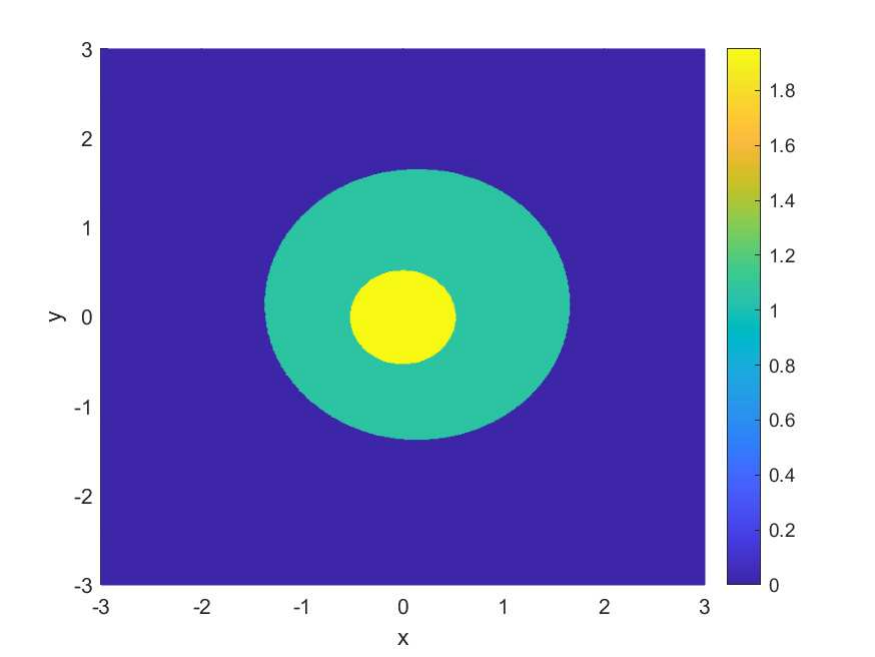}}\\
\subfloat[ $i=5$]{\includegraphics[width=0.3\textwidth]
                   {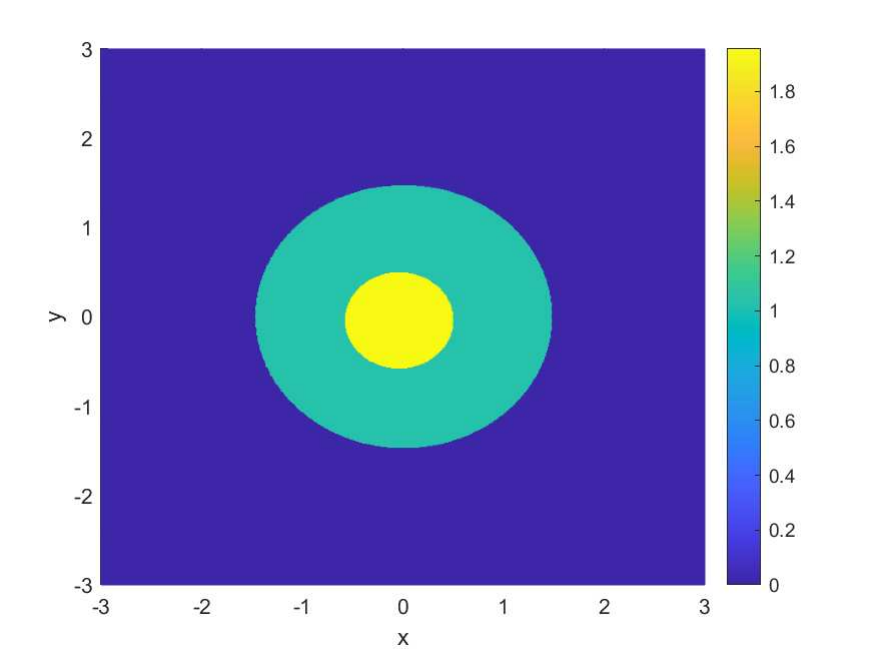}}
\subfloat[ $exact\ solution $]{\includegraphics[width=0.3\textwidth]
                   {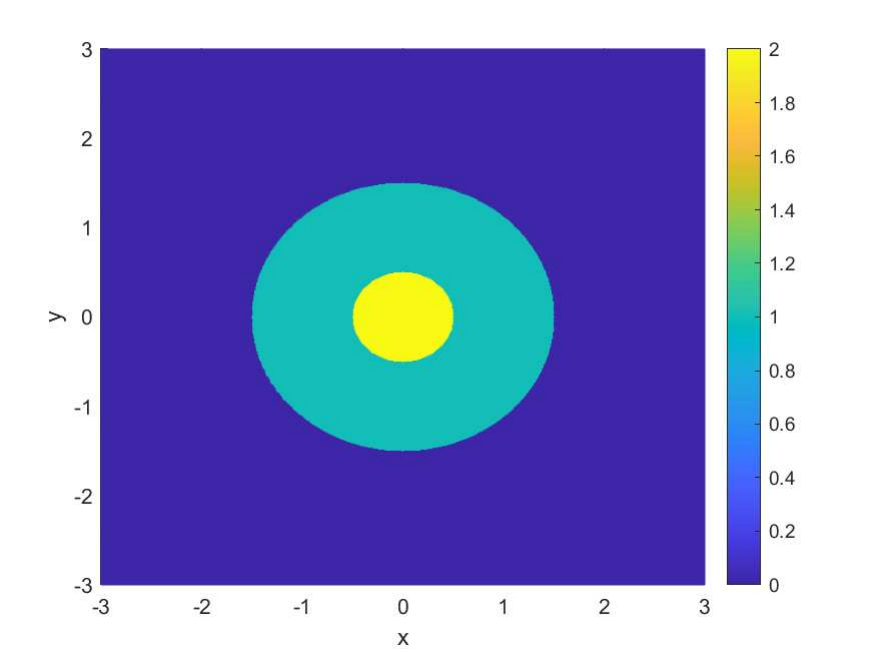}}
\subfloat[ $e_r $]{\includegraphics[width=0.3\textwidth]
                   {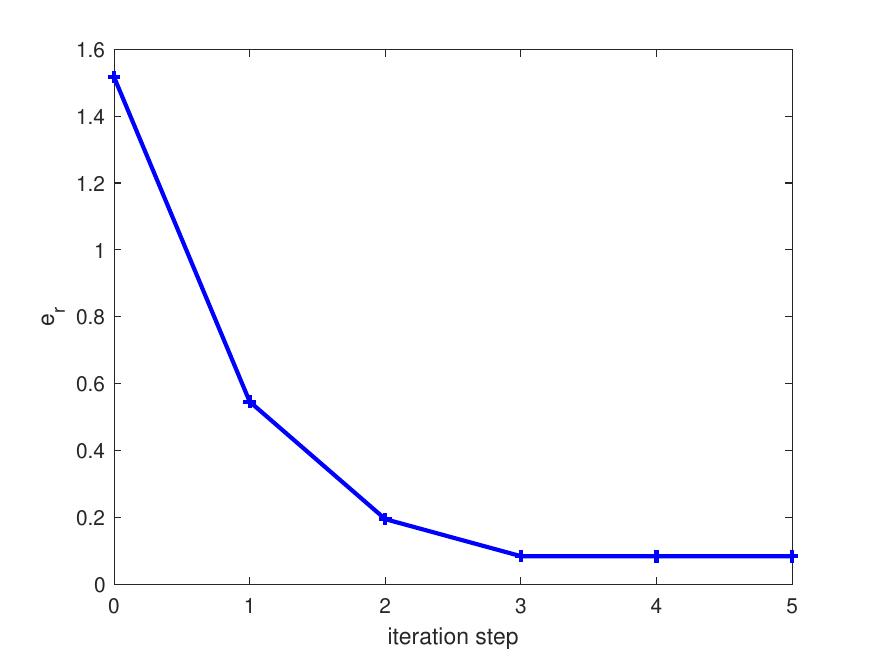}}

\caption{\label{circle_n}  (a) (b) (c) (d) reconstructed solutions at different iteration steps, (e) exact solution, (f) $e_r$ for different iteration steps $i$. }
\end{figure}

Next, we present numerical examples to validate the accurate reconstruction of the source term in three dimensions.

\begin{exm} \label{exm3}
We consider a simple case where $\omega=B_1(0)$ is a ball and the strength of its source is $\varphi=1$.
\end{exm}

In Example \ref{exm3}, the function $D(x)$ takes the form of a constant within $\Omega$, which represents the heart as the material. Specifically, we have $\mu_a = 0.011$ and $\mu_s' = 1.096$. We select $g=x$. Our objective is to demonstrate that our numerical experiments successfully achieve a distinct reconstruction of the source term, even in a non-dark environment. The regularization parameter is given by $\beta=0.8, i_0=8$, and the initial  guess $q^0$ is chosen for a ball with centers at $(0.5, 0.5,0.5)$ and a radius of 0.5, and the intensity of the source within $\omega$ is set to $\varphi = 0.1$.
Similar to the numerical computations in two dimensions, we assess the accuracy between the approximate solution and the exact solution. This evaluation is depicted in Figure \ref{sphere}. In Figure \ref{sphere} (a) and  (d), we present the initial guess $\omega$ and the initial guess $\omega$ and $\varphi$ in the $z = 0$ plane, respectively. The results of the final reconstruction are displayed in Figure \ref{sphere} (b) and  (e), where Figure \ref{sphere} (b) represents the 3D stereogram of the final iteration of $\omega$, and Figure \ref{sphere} (e) showcases the reconstructed result of $\omega$ and $\varphi$ in the $z = 0$ plane. Additionally, the exact solution is illustrated in Figures \ref{sphere} (c) and  (f). It can be observed that the reconstructed result obtained through iteration closely aligns with the true solution, indicating a good fit.

\begin{figure}
\centering
\subfloat[]{\includegraphics[width=0.3\textwidth]
                   {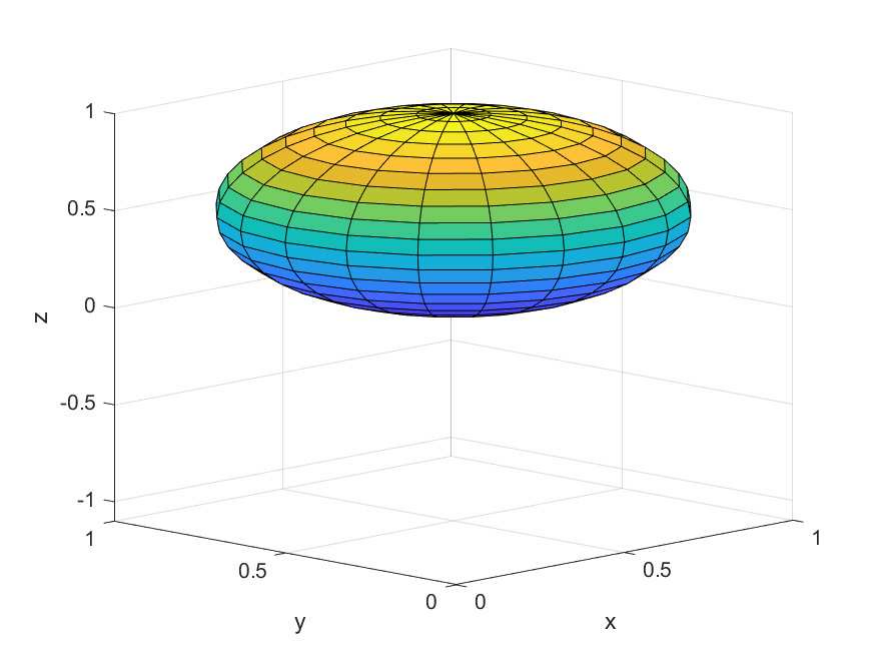}}
\subfloat[]{\includegraphics[width=0.3\textwidth]
                   {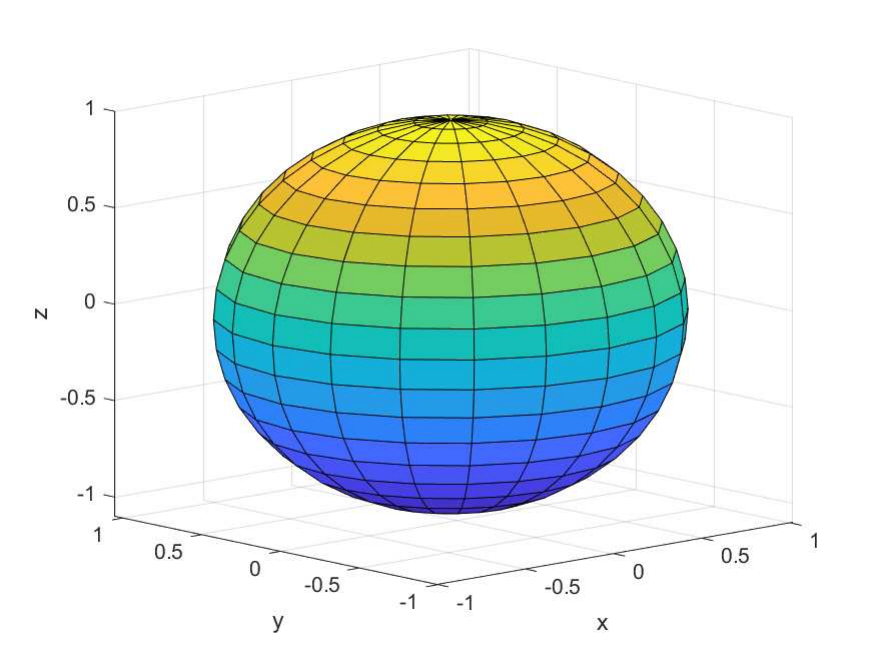}}
\subfloat[]{\includegraphics[width=0.3\textwidth]
                   {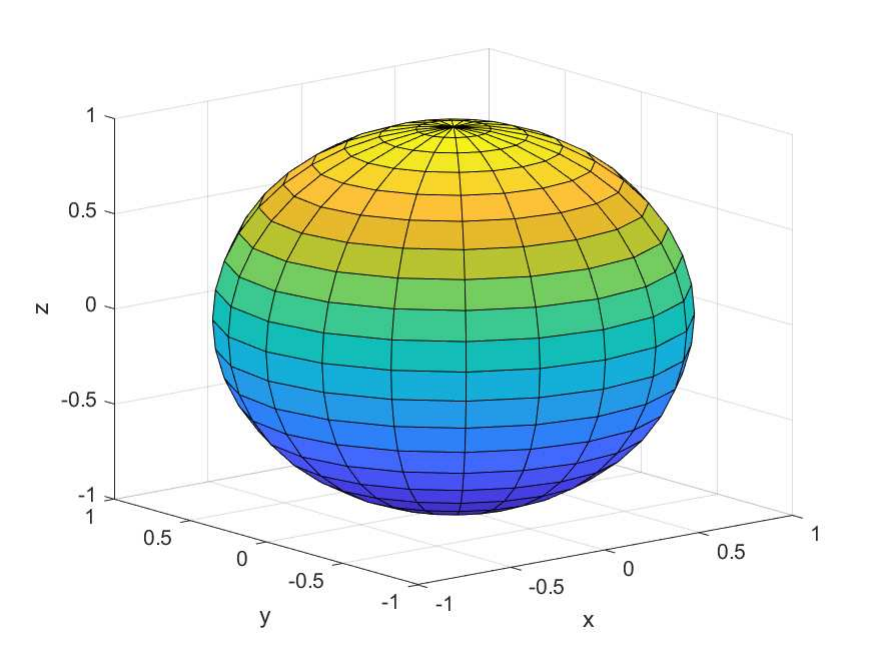}}\\
\subfloat[ ]{\includegraphics[width=0.3\textwidth]
                   {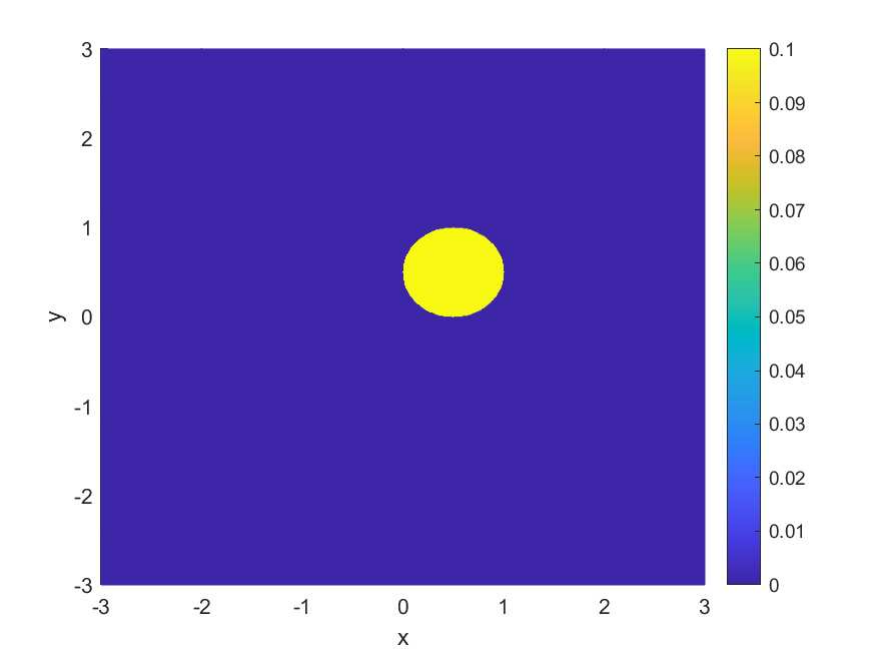}}
\subfloat[]{\includegraphics[width=0.3\textwidth]
                   {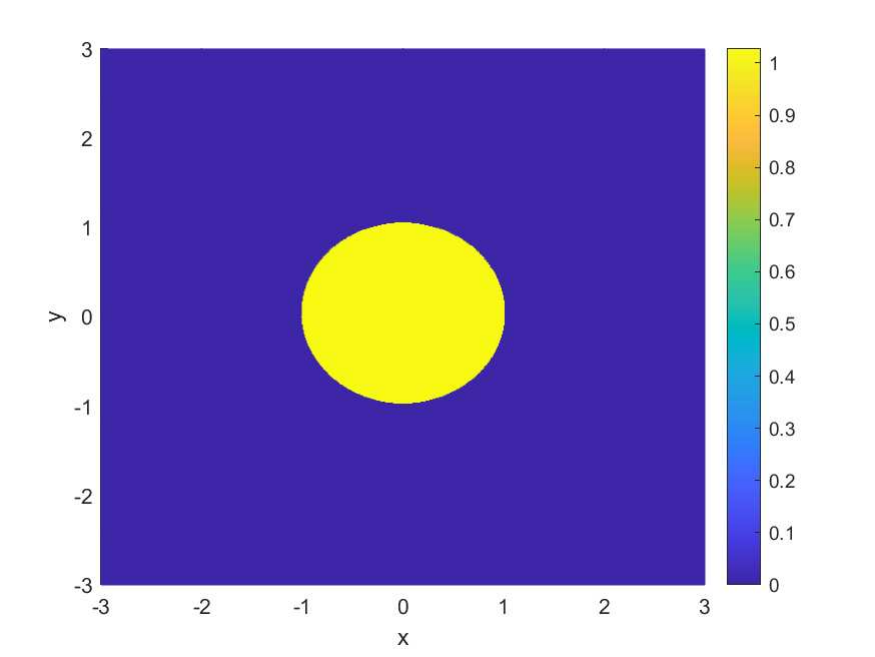}}
\subfloat[]{\includegraphics[width=0.3\textwidth]
                   {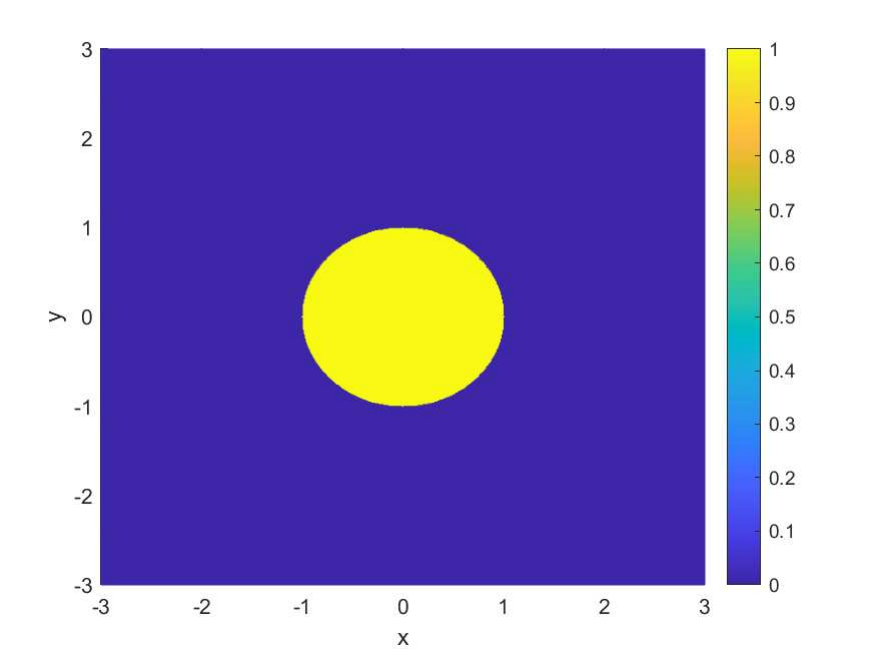}}
\caption{\label{sphere} (a) Initial iteration of $\omega$, (b) reconstructed $\omega$, (c)  exact $\omega$, (d) initial iteration of $\omega$ and $\varphi$ in $z=0$ plane, (e) reconstructed $\omega$ and $\varphi$ in $z=0$ plane, (f) the exact $\omega$ and $\varphi$ in $z=0$ plane. }
\end{figure}
\begin{exm} \label{exm4}
We consider $\omega$ to be an ellipsoid with center at  $(0, 0, 0)$, where the equatorial radius along the $x$ and $y$ axes is 1, 2, respectively, and the polar radius (along the $z$-axis) is 1, and the source $\varphi=1$ within $\omega$.
\end{exm}

In Example \ref{exm4}, the material in the background region $\Omega$ is a heart characterised by $\mu_a = 0.011$ and $\mu_s' = 1.096$. The input function $g=x$. The regularization parameters are specified as $\beta=0.6$ and $i_0=1$, and the initial guess $\omega$ is set to an ellipsoid with center at $(0.5,0.5,0.5)$, where the equatorial radius along the $x$ and $y$ axes is 0.3, 1.5 respectively, and the polar radius is 0.3, and the $\varphi=0.5$. We conduct a thorough evaluation of the accuracy by comparing the approximate solution with the exact solution, the results of this evaluation are presented in Figure \ref{ellipsoid}. Notably, the reconstructed result obtained through iterative processes demonstrates a remarkable alignment with the true solution.

\begin{figure}
\centering
\subfloat[ ]{\includegraphics[width=0.3\textwidth]
                   {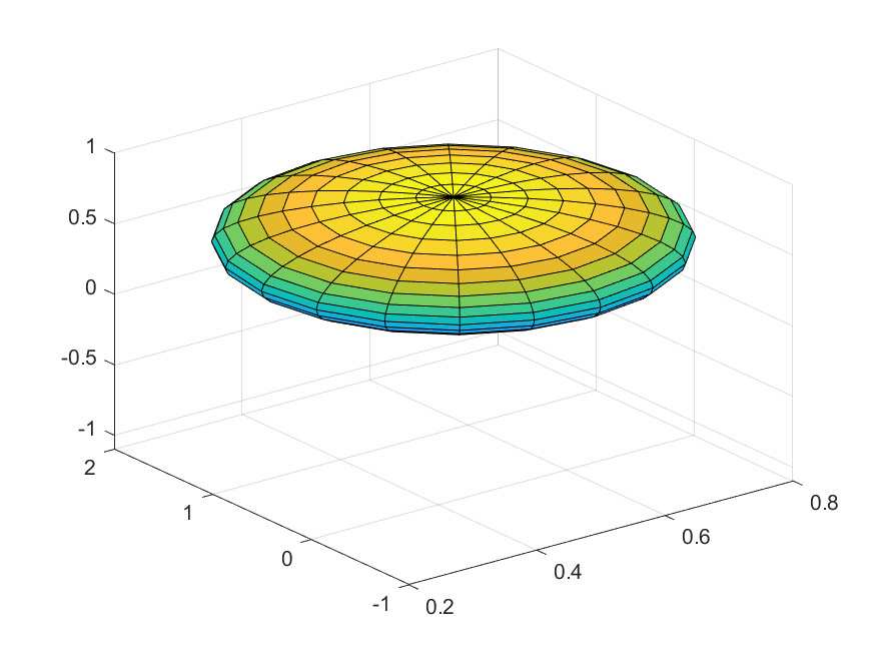}}
\subfloat[]{\includegraphics[width=0.3\textwidth]
                   {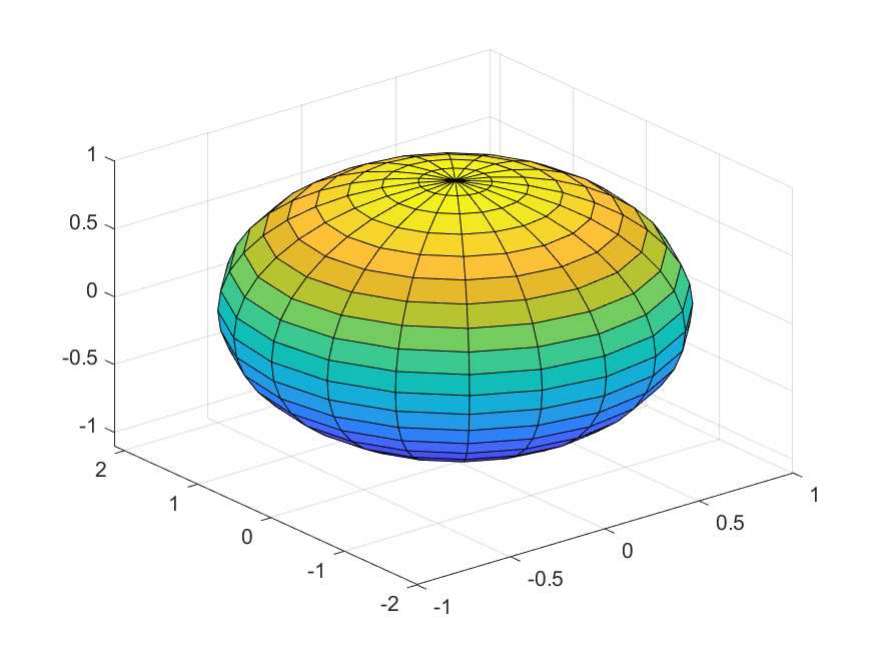}}
\subfloat[]{\includegraphics[width=0.3\textwidth]
                   {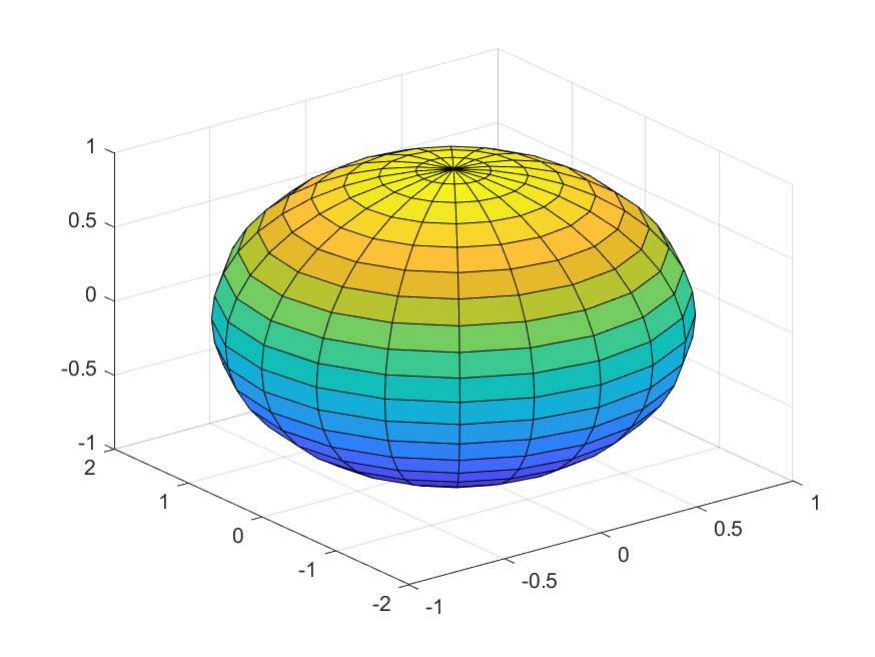}}\\
\subfloat[]{\includegraphics[width=0.3\textwidth]
                   {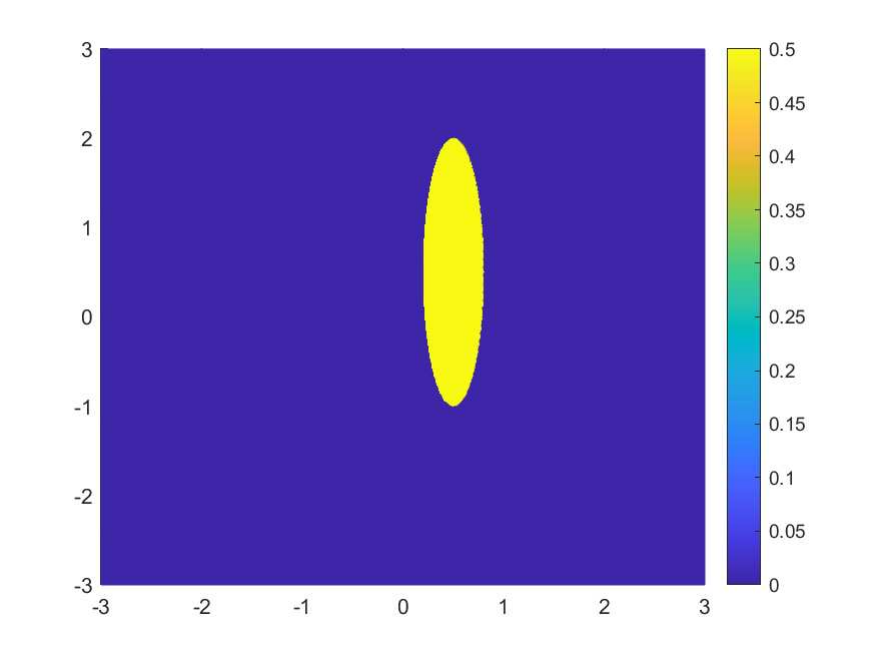}}
\subfloat[]{\includegraphics[width=0.3\textwidth]
                   {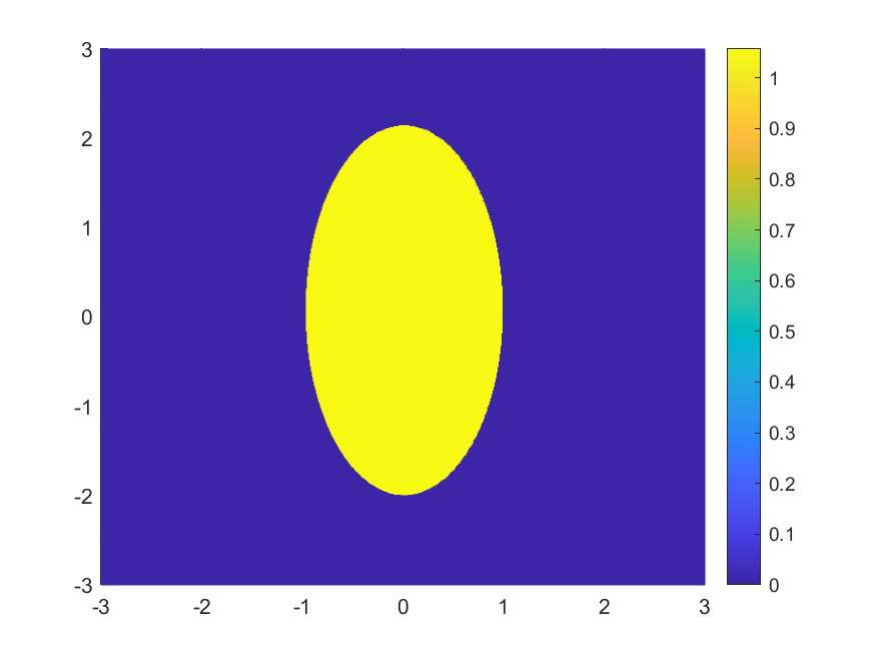}}
\subfloat[]{\includegraphics[width=0.3\textwidth]
                   {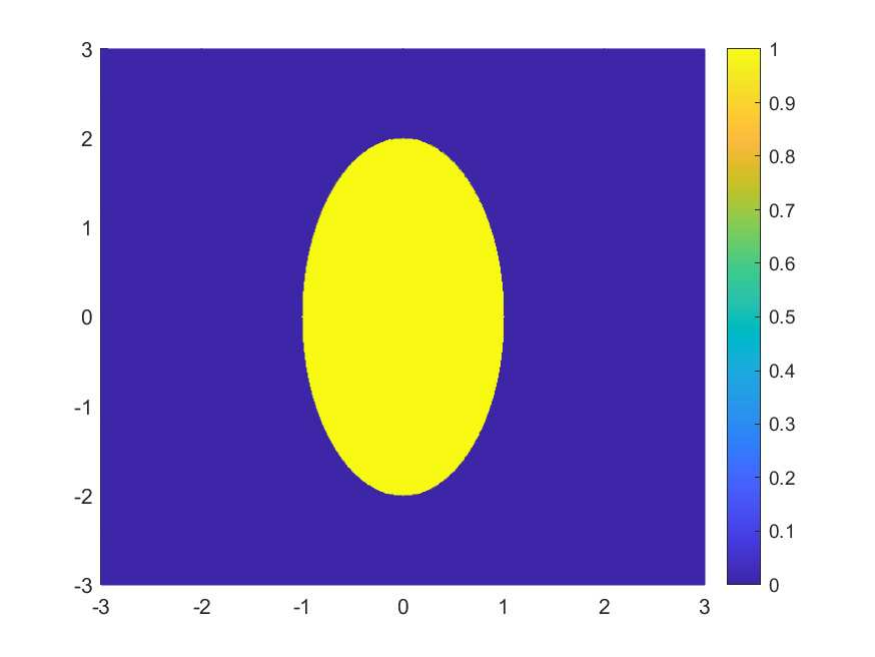}}
\caption{\label{ellipsoid}(a) Initial iteration of $\omega$, (b) reconstructed $\omega$, (c) exact $\omega$, (d) initial iteration of $\omega$ and $\varphi$ in $z=0$ plane, (e) reconstructed $\omega$ and $\varphi$ in $z=0$ plane, (f) the exact $\omega$ and $\varphi$ in $z=0$ plane. }
\end{figure}

\subsection{Polyhedral Domains}

In this subsection, we focus on the scenario where $\omega$ is a polyhedral domain. The input function $g^- = 0$ in all the subsequent examples indicates that the investigation is conducted in a dark environment.

\begin{exm} \label{Pexm1}
In this example, $\omega$  is a rectangle with a length $1$ and a width $1/2$. We denote the four vertices of the rectangle as $\bar{A}, \bar{B}, \bar{C}$, and $\bar{D}$. The coordinates of $\bar{A}$ are $(0, 0)$, the coordinates of $\bar{B}$ are $(\sqrt{3}/4,1/4)$, the coordinates of $\bar{C}$ are $(-1/2,\sqrt{3}/2)$, and the coordinates of the corresponding $\bar{D}$ can be deduced accordingly. The intensity of the source is set to $\varphi=3$.
\end{exm}

In  Example \ref{Pexm1}, the material in the background domain $\Omega$ is lung, characterized by $\mu_a= 0.023,\mu_s'= 2$ in $\Omega$. The regularization parameters $i_0=4, \beta=0.6$ and the initial guess of $\omega$ is a rectangle with a length $\sqrt{0.5}$ and a width $\sqrt{0.5}/2$, and the coordinates of $\bar{A}$ are $(0.5, 0.5)$, the coordinates of $\bar{B}$ are $(1,1)$,  and the intensity of the source is set to $\varphi=2.7$. We examine the accuracy between the reconstructed solution and the true solution in Figure \ref{ju}. Figures \ref{ju} 
 (a)-(d) display the evolution of the reconstructed solution at different iteration steps $i$. Furthermore, Figure \ref{ju} (f) demonstrates the relative error $e_r$ variation as a function of the number of iteration steps. It is worth noting that the approximate solution we obtained at the maximum number of iteration step has a relative error of 0.0458 with respect to the exact solution.  The excellent agreement between the reconstructed solution and the exact solution clearly demonstrates the accurate reconstruction of the location, shape, and size of the internal sources through boundary measurements \eqref{Neumann}, even in the case of a polygonal region for $\omega$.

\begin{figure}
\centering
\subfloat[$i=0$ ]{\includegraphics[width=0.3\textwidth]
                   {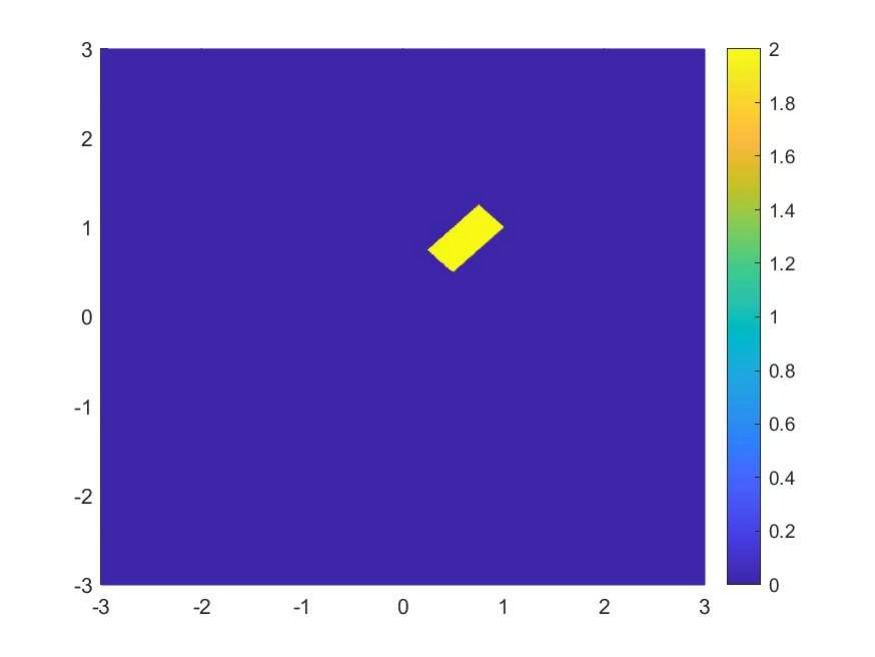}}
\subfloat[$i=5$ ]{\includegraphics[width=0.3\textwidth]
                   {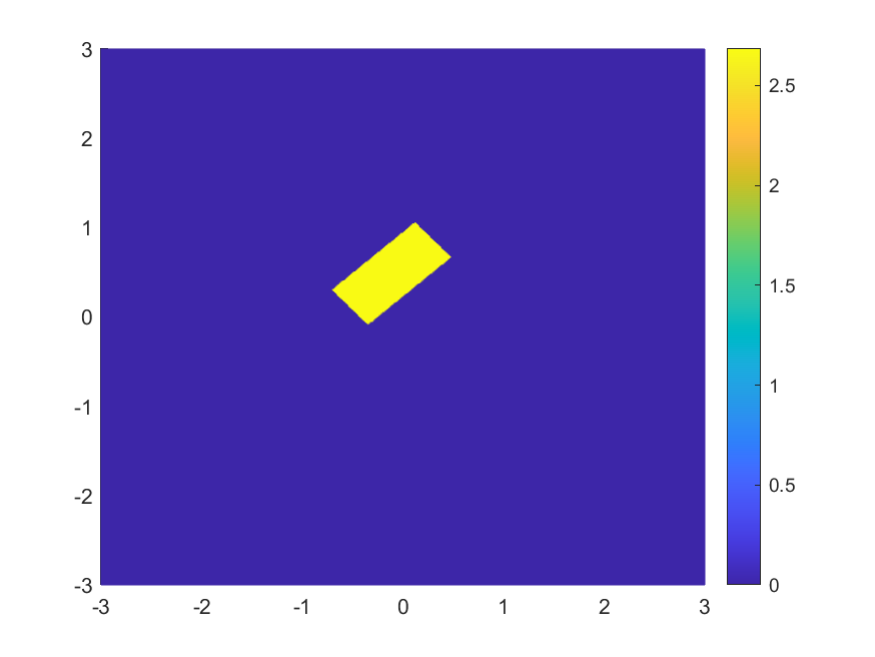}}
\subfloat[ $i=12$]{\includegraphics[width=0.3\textwidth]
                   {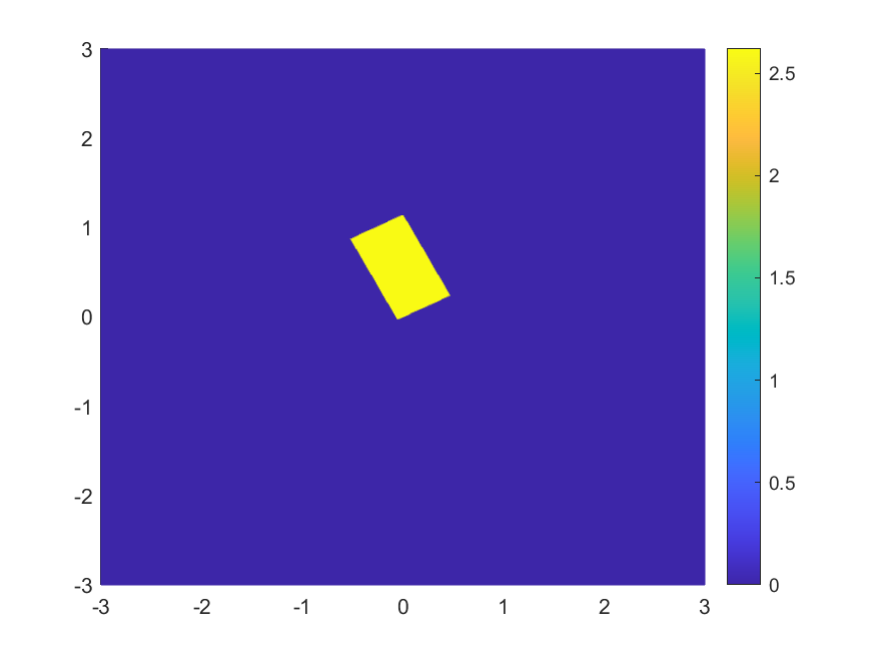}}\\
\subfloat[ $i=20$]{\includegraphics[width=0.3\textwidth]
                   {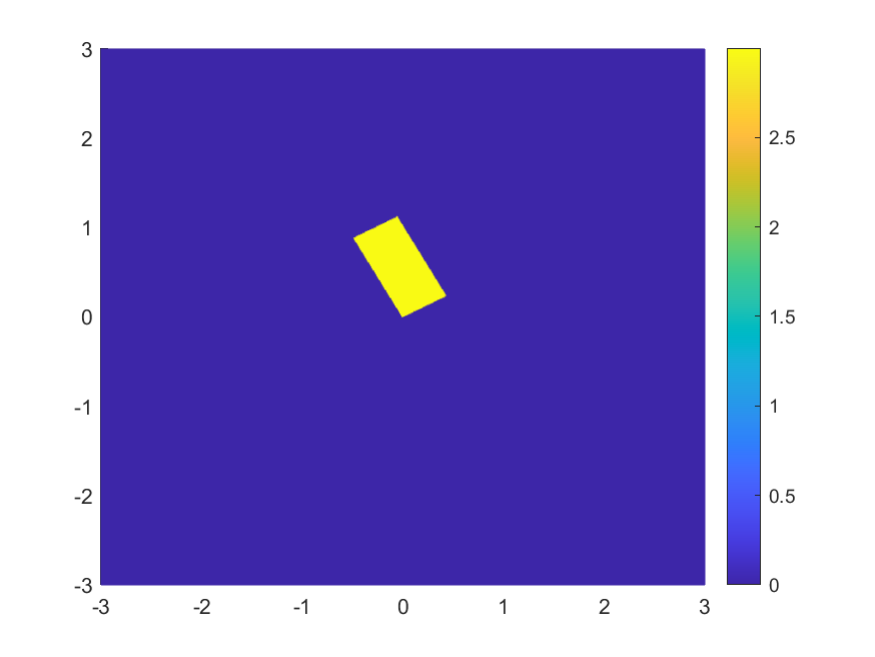}}
\subfloat[ $exact\ solution $]{\includegraphics[width=0.3\textwidth]
                   {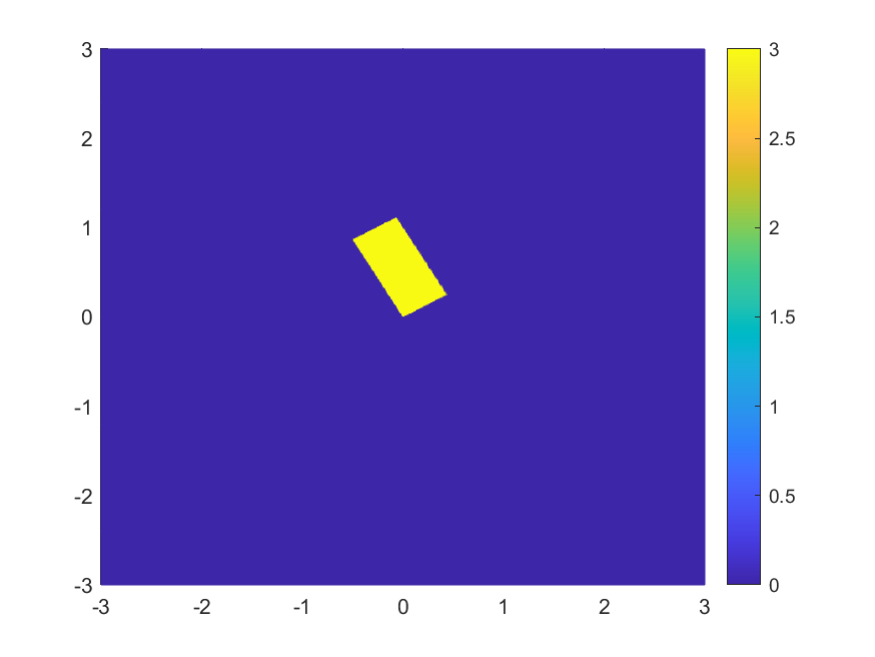}}
\subfloat[ $e_r $]{\includegraphics[width=0.3\textwidth]
                   {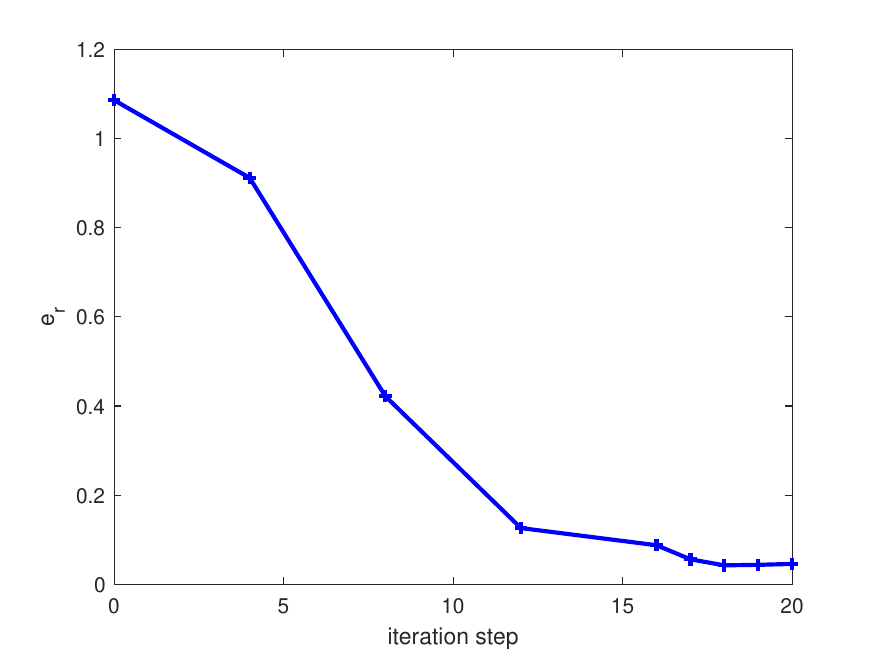}}

\caption{\label{ju}  (a) (b) (c) (d) reconstructed solutions at different iteration steps, (e) exact solution, (f) $e_r$ for different iteration steps $i$. }
\end{figure}

\begin{exm} \label{Pexm2}
In the second example, $\omega$ consists of a combination of a circular region with a radius of 1 and a triangular region. The circular region covers 3/4 of the circle $B_1(0)$, and the coordinates of the vertices of the unconnected vertices of the triangular are $A=(-2, 0)$, and $\varphi=3$ in the support set $\omega$.
\end{exm}

In Example \ref{Pexm2}, the material in the background domain $\Omega$ is heart, characterized by $\mu_a= 0.01,\mu_s'= 1.096$ in $\Omega$. The regularization parameter $i_0=8, \beta=0.7$ and the initial iteration for the unconnected vertices of the triangular is $A=(-1,1)$ and $\varphi=10$. Figure \ref{cor} provides a comprehensive analysis of the accuracy between the reconstructed solution and the exact solution. Figures \ref{cor} (a)-(d) showcase the iterative evolution of the reconstructed solution. Additionally, Figure \ref{cor} (f) illustrates the relative error variation as the number of iteration steps increases.
It is important to highlight that after 7 iterations, the reconstructed solution closely approximates the exact solution. This means that we can accurately determine the location of the corner points and the strength of the source term, even when $\omega$ is the shape of the corona. This observation aligns with the findings of Theorem \ref{thm:mainCorona}.

\begin{exm} \label{Pexm3}
We examine a three-dimensional case where $\omega$ is a cube with a side length of 2 and its origin located at coordinates $(0,0,0)$. Within $\omega$, the source term $\varphi$ is set to a constant value of 3.
\end{exm}

In Example \ref{Pexm3}, the material in the background domain $\Omega$ is heart, characterized by $\mu_a= 0.01,\mu_s'= 1.096$ in $\Omega$. The regularization parameter $i_0=3, \beta=0.68$ and the initial guess of $\omega $ is also a cube with side lengths of 1, the origin located at coordinates $(0.2,0.2,0,2)$, and the initial guess of strength for the source is 2.5. Figure \ref{cub} provides a comprehensive analysis of the accuracy between the reconstructed solution and the exact solution, employing a methodology similar to the two-dimensional examples. The close alignment between the reconstructed solution and the exact solution serves as confirmation of the findings.

\begin{figure}
\centering
\subfloat[$i=0$ ]{\includegraphics[width=0.3\textwidth]
                   {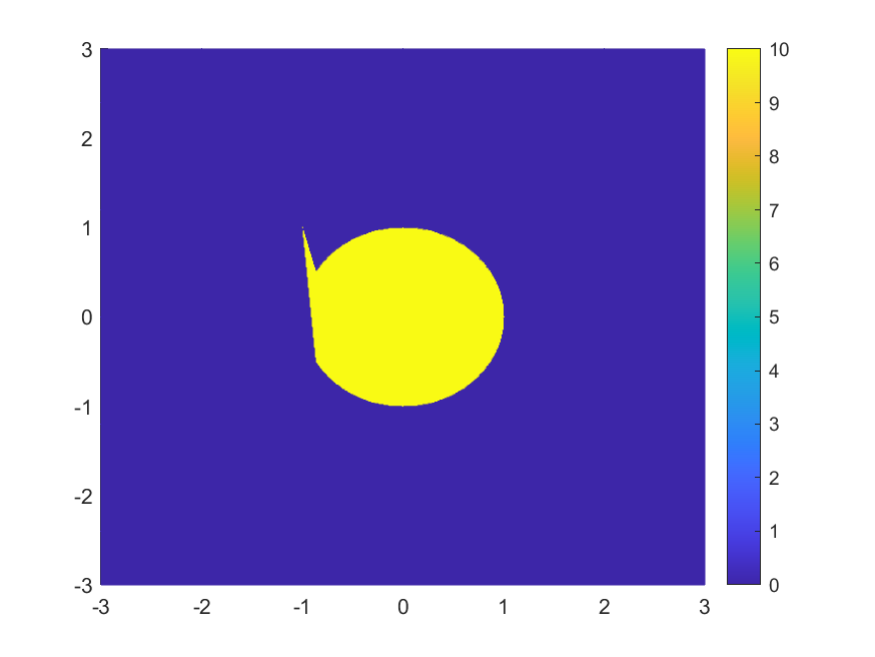}}
\subfloat[$i=1$ ]{\includegraphics[width=0.3\textwidth]
                   {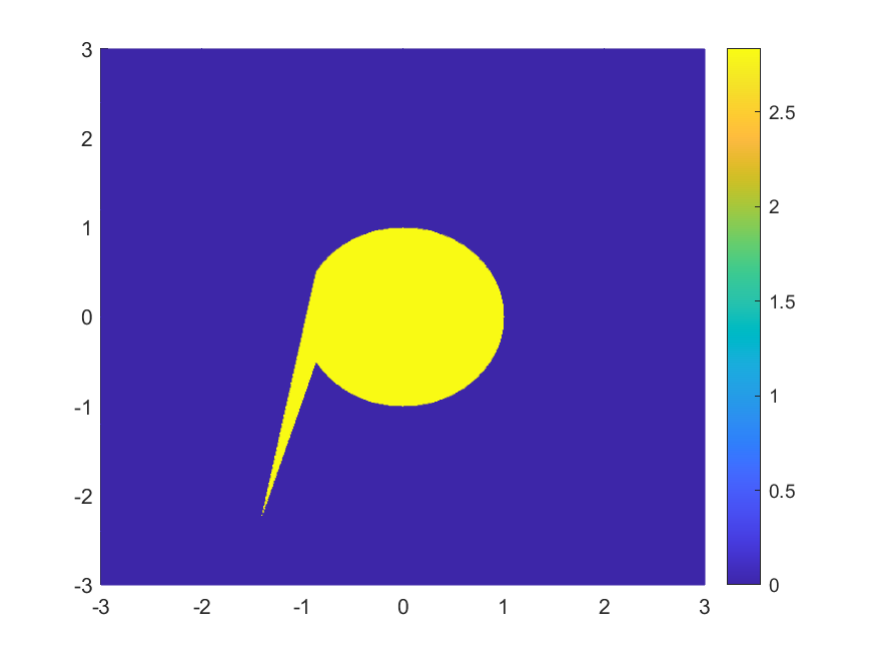}}
\subfloat[ $i=2$]{\includegraphics[width=0.3\textwidth]
                   {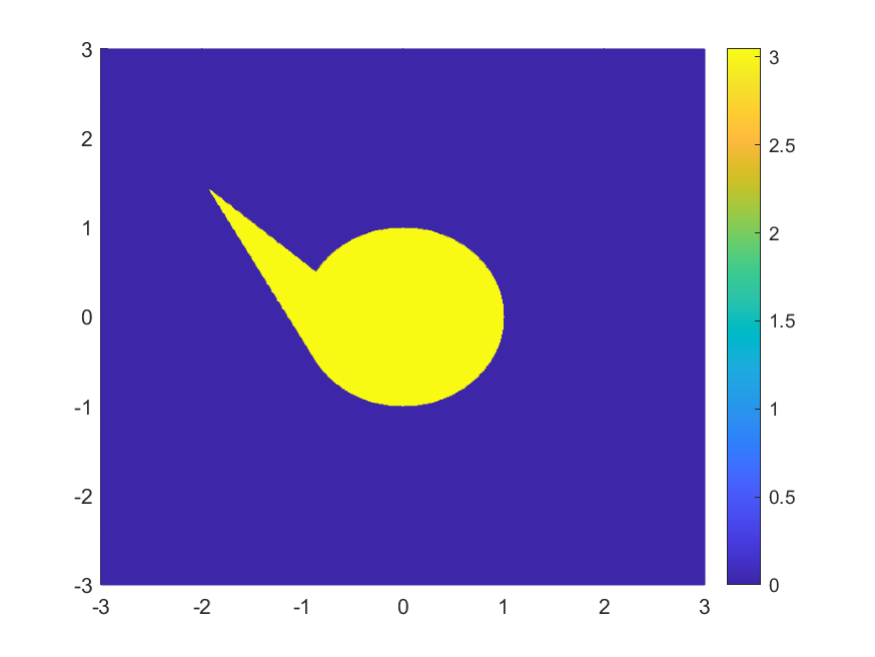}}\\
\subfloat[ $i=7$]{\includegraphics[width=0.3\textwidth]
                   {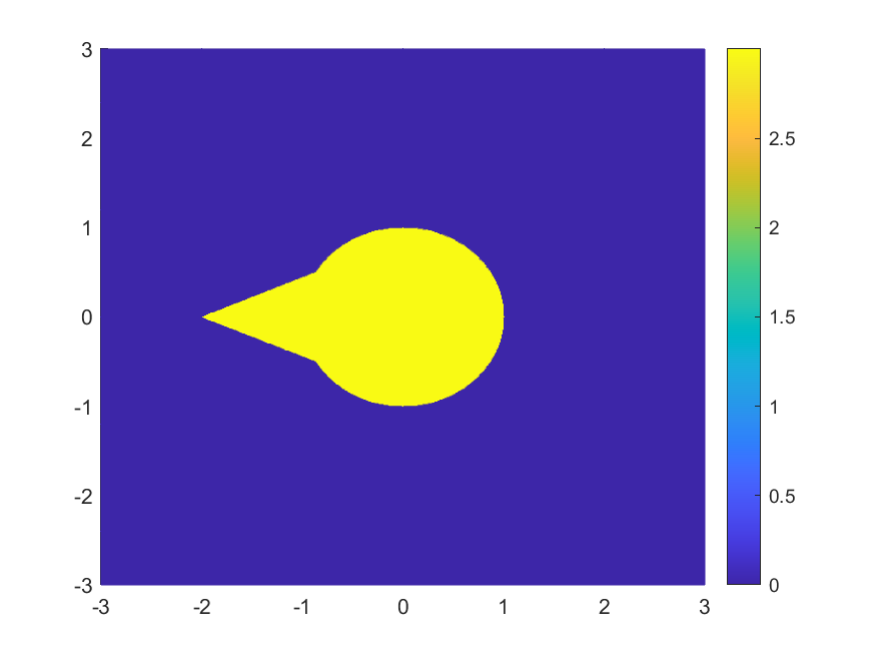}}
\subfloat[ $exact\ solution $]{\includegraphics[width=0.3\textwidth]
                   {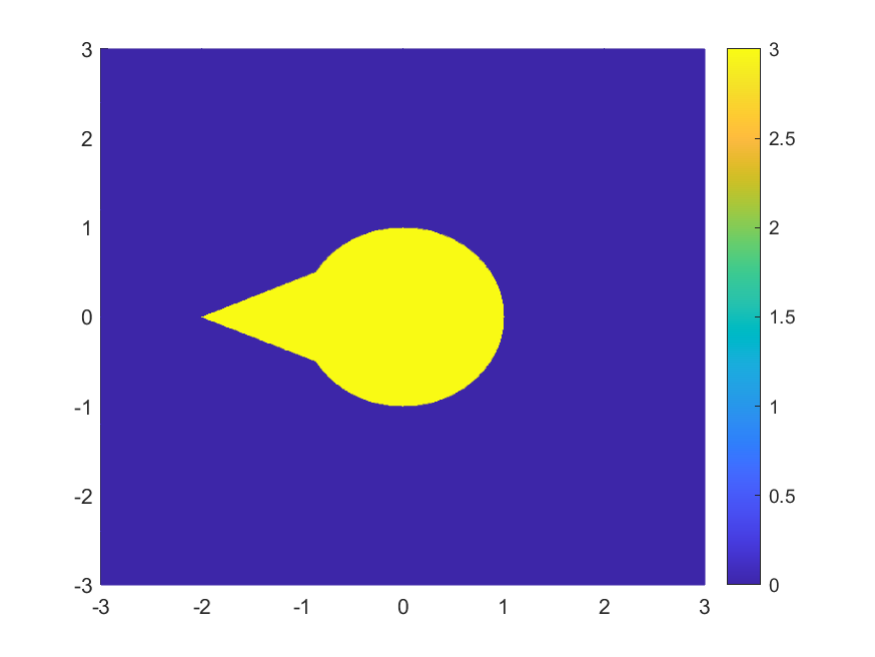}}
\subfloat[ $e_r $]{\includegraphics[width=0.3\textwidth]
                   {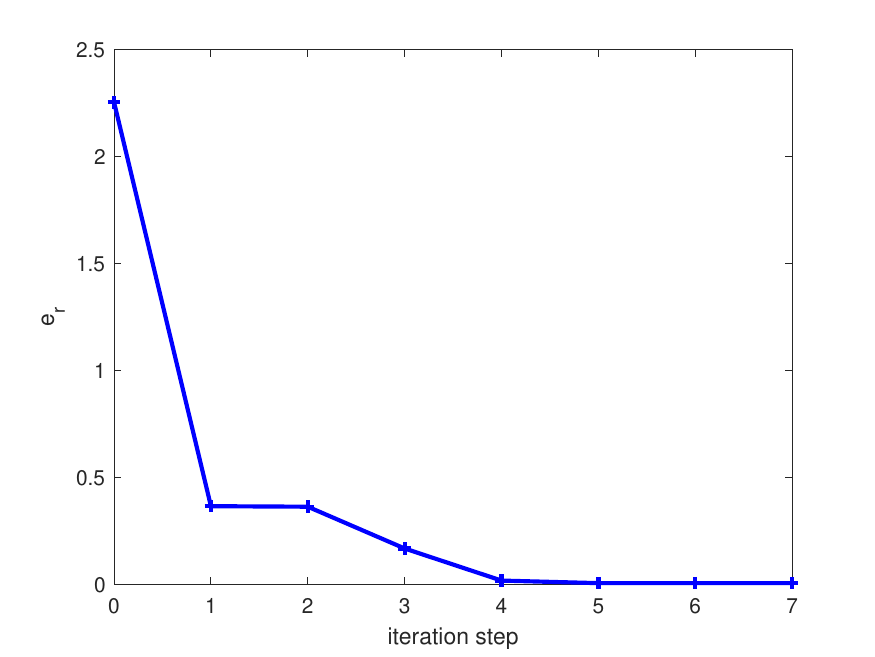}}

\caption{\label{cor}  (a) (b) (c) (d) reconstructed solutions at different iteration steps, (e) exact solution, (f) $e_r$ for different iteration steps $i$. }
\end{figure}

\begin{figure}
\centering
\subfloat[]{\includegraphics[width=0.3\textwidth]
                   {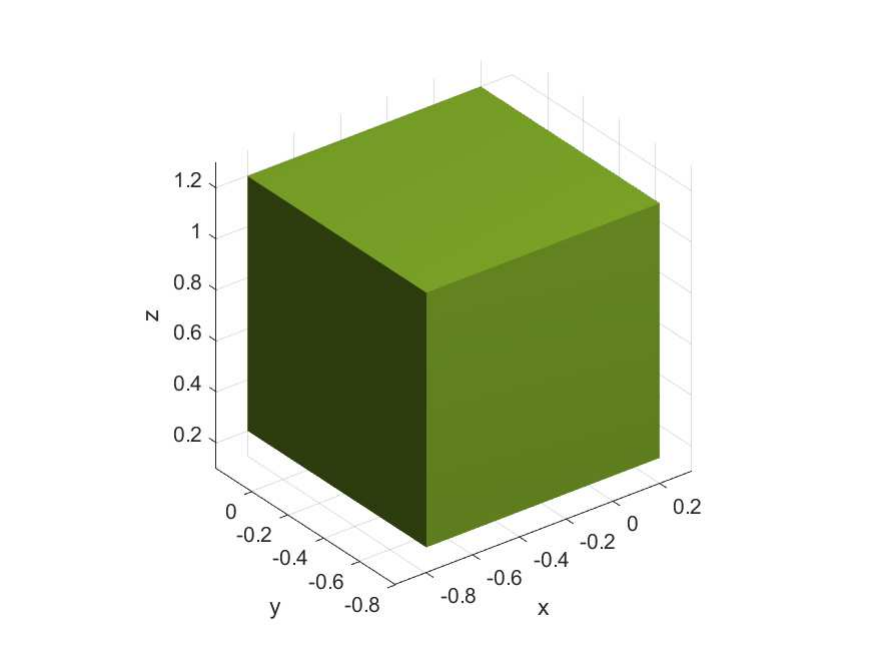}}
\subfloat[ ]{\includegraphics[width=0.3\textwidth]
                   {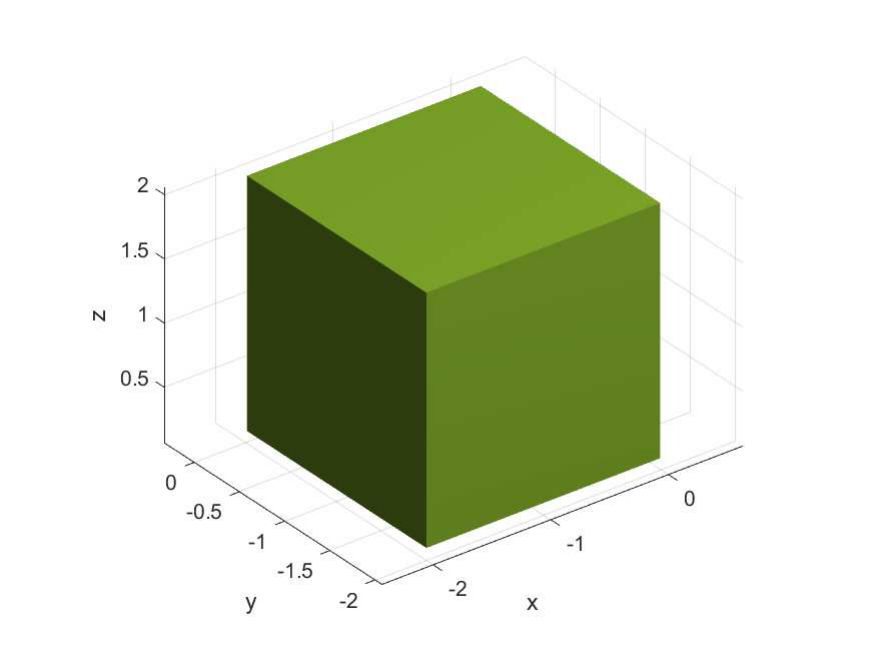}}
\subfloat[]{\includegraphics[width=0.3\textwidth]
                   {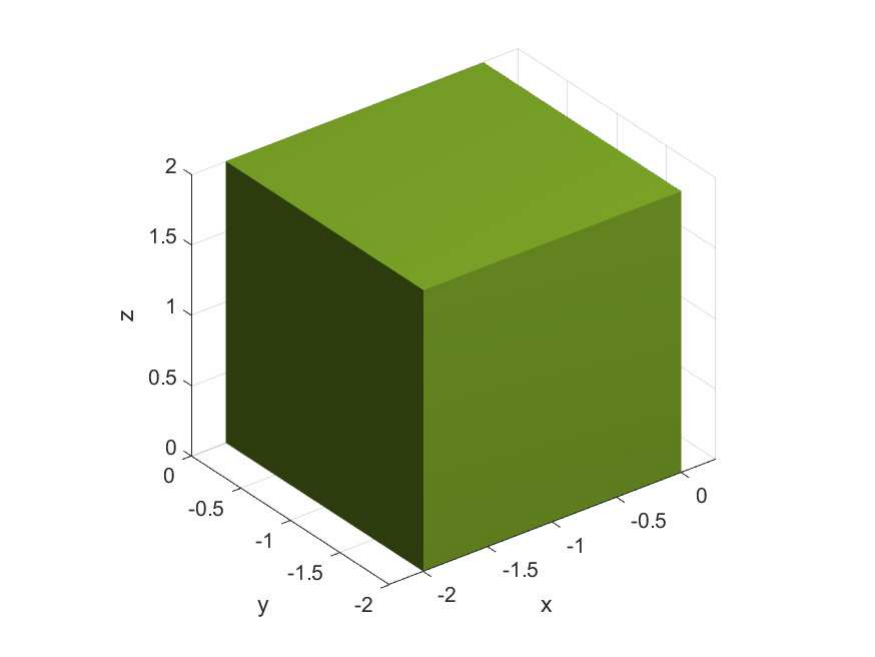}}\\
\subfloat[]{\includegraphics[width=0.3\textwidth]
                   {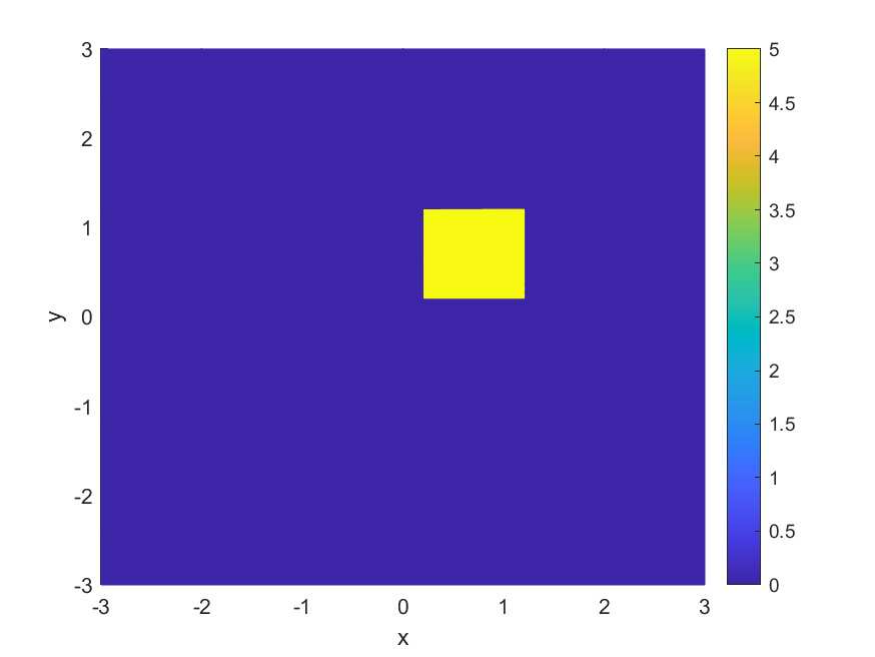}}
\subfloat[]{\includegraphics[width=0.3\textwidth]
                   {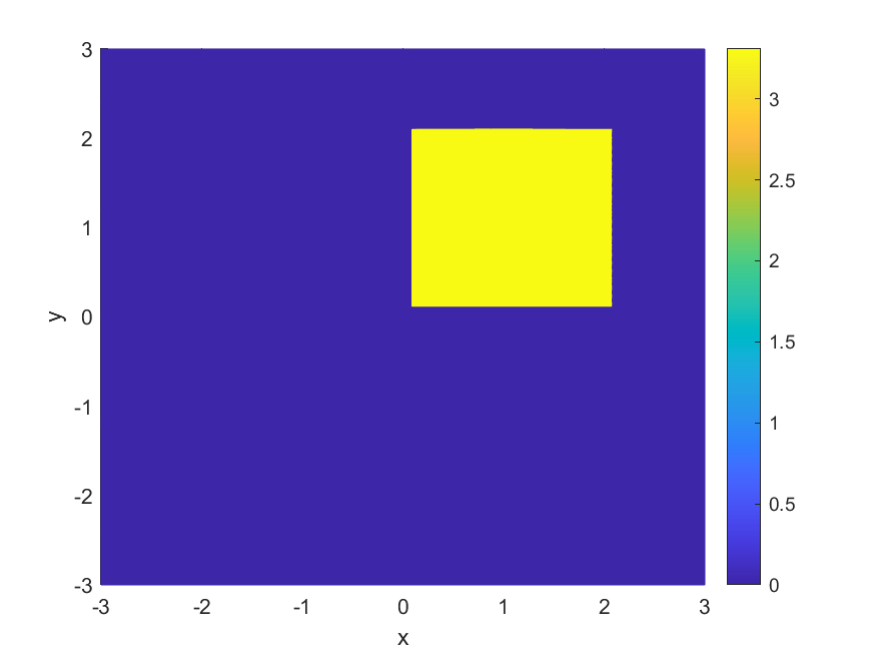}}
\subfloat[]{\includegraphics[width=0.3\textwidth]
                   {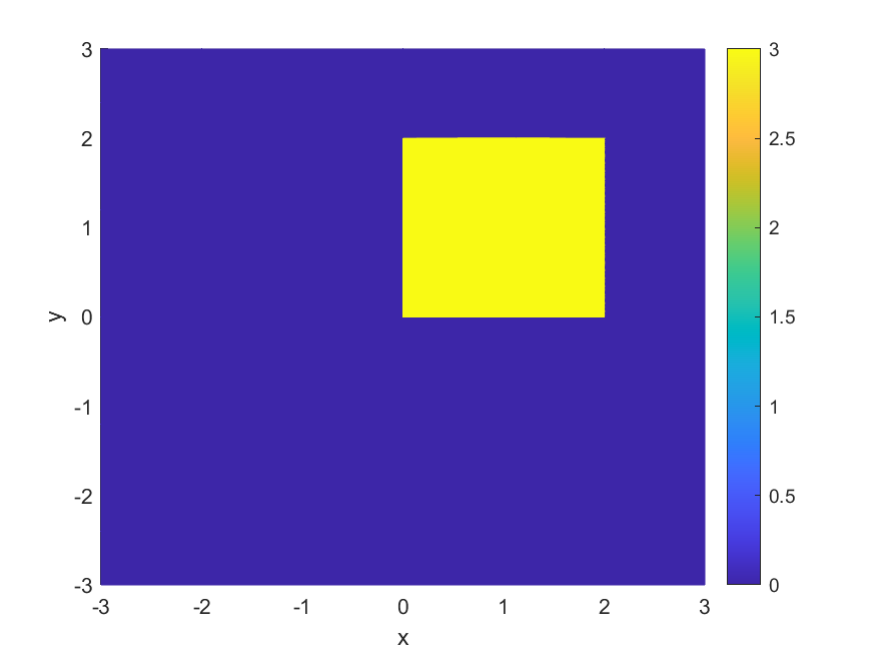}}
\caption{\label{cub} (a) Initial iteration of $\omega$, (b) reconstructed $\omega$, (c) exact $\omega$, (d) initial iteration of $\omega$ and $\varphi$ in $z=0$ plane, (e) reconstructed $\omega$ and $\varphi$ in $z=0$ plane, (f) the exact $\omega$ and $\varphi$ in $z=0$ plane.  }
\end{figure}

\bibliographystyle{siamplain}
%\printbibliography
\bibliography{ref.bib,refInversePb,refnumerical.bib}
\end{document}